\documentclass[a4paper,11pt,reqno]{amsart}
\usepackage{amsmath}
\usepackage{amsthm,enumerate}

\usepackage{graphicx}
\usepackage{amssymb}

\usepackage{appendix}
\usepackage{fancyhdr}

\usepackage{calc}
\usepackage{url}

\usepackage[text={6in,9in},centering]{geometry}

\usepackage{srcltx, inputenc}

\usepackage[T1]{fontenc}

\usepackage[usenames,dvipsnames]{color}

\makeatletter
\def\cleardoublepage{\clearpage\if@twoside \ifodd\c@page\else%
         \hbox{}%
     \thispagestyle{empty}
     \newpage%
     \if@twocolumn\hbox{}\newpage\fi\fi\fi}
\makeatother

\let\cleardoublepage\clearpage

\newtheorem{thm}{Theorem}[section]

\newtheorem{oss}[thm]{Remark}
\numberwithin{equation}{section}

\usepackage{xcolor}

\newtheorem{theorem}{Theorem}[section]
\newtheorem{proposition}[theorem]{Proposition}

\newtheorem{lemma}[theorem]{Lemma}

\newtheorem{definition}[theorem]{Definition}

\usepackage{pdfsync}

\begin{document}

\title[Fast diffusion on noncompact Riemannian manifolds]{Fast diffusion on noncompact manifolds: \\ well-posedness theory and connections \\ with semilinear elliptic equations}

\author {Gabriele Grillo, Matteo Muratori, Fabio Punzo}

\address {Gabriele Grillo: Dipartimento di Matematica, Politecnico di Milano, Piaz\-za Leo\-nar\-do da Vinci 32, 20133 Milano, Italy}
\email{gabriele.grillo@polimi.it}

\address{Matteo Muratori: Dipartimento di Matematica, Politecnico di Milano, Piaz\-za Leo\-nar\-do da Vinci 32, 20133 Milano, Italy}
\email{matteo.muratori@polimi.it}

\address{Fabio Punzo: Dipartimento di Matematica, Politecnico di Milano, Piaz\-za Leo\-nar\-do da Vinci 32, 20133 Milano, Italy}
\email{fabio.punzo@polimi.it}


\begin{abstract}
We investigate the well-posedness of the fast diffusion equation (FDE) in a wide class of noncompact Riemannian manifolds. Existence and uniqueness of solutions for globally integrable initial data was established in \cite{BGV}. However, in the Euclidean space, it is known from Herrero and Pierre \cite{HP} that the Cauchy problem associated with the FDE is well posed for initial data that are merely in $ L^1_{\mathrm{loc}} $. We establish here that such data still give rise to global solutions on general Riemannian manifolds. If, in addition, the radial Ricci curvature satisfies a suitable pointwise bound from below (possibly diverging to $-\infty$ at spatial infinity), we prove that also uniqueness holds, for the same type of data, in the class of strong solutions. Besides, under the further assumption that the initial datum is in $L^2_{\mathrm{loc}}$ and nonnegative, a minimal solution is shown to exist, and we are able to establish uniqueness of purely (nonnegative) distributional solutions, which to our knowledge was not known before even in the Euclidean space. The required curvature bound is in fact sharp, since on model manifolds it turns out to be equivalent to stochastic completeness, and it was shown in \cite{GIM} that uniqueness for the FDE fails even in the class of bounded solutions on manifolds that are not stochastically complete. Qualitatively this amounts to asking that the curvature diverges at most quadratically at infinity. A crucial ingredient of the uniqueness result is the proof of nonexistence of distributional subsolutions to certain semilinear elliptic equations with power nonlinearities, of independent interest.

\end{abstract}

\maketitle

\bigskip

\section{Introduction} \setcounter{equation}{0}
We study existence and uniqueness of solutions to the Cauchy problem for the following nonlinear parabolic equation, known as \it fast diffusion equation \rm (FDE):
\begin{equation}\label{cauchy}
\begin{cases}
u_t=\Delta u^m & \text{in } M \times (0, +\infty) \, , \\
u=u_0 & \text{on } M \times \{0\} \, ,
\end{cases}
\end{equation}
where $ m \in (0,1) $, the initial datum $u_0$ belongs to a suitable class that will be specified below, $M$ is a complete, connected, noncompact $n$-dimensional Riemannian manifold and $ \Delta $ is the Laplace-Beltrami operator on $M$ (Laplacian for short). When dealing with sign-changing solutions, we adopt the usual convention $u^m:=\operatorname{sign}(u) |u|^m$. 
Let $o\in M$ be a fixed reference point, and let $r(x)$ denote the geodesic distance between $x$ and $o$.
In most of our results we will assume that the {\em radial} Ricci curvature with respect to $o$ satisfies the following bound from below:
\begin{equation}\label{eq1b}
\mathrm{Ric}_o(x) \ge - (n-1) \, \frac{\psi''(r(x))}{\psi(r(x))}\qquad  \forall x \in M \setminus \left(  \{ o \} \cup \operatorname{cut}(o)  \right)  ,
\end{equation}
where $  \operatorname{cut}(o) $ is the cut locus of $ o $, for some function $\psi$ such that
$$
\psi\in C^\infty((0, \infty))\cap C^1([0, \infty)) \, , \quad \psi' \geq 0 \, , \ \, \psi(0)=0 \, , \ \, \psi'(0)=1
$$
and
\begin{equation}\label{eq1c}
\int_0^{\infty} \frac{\int_0^r \psi(\rho)^{n-1} \, d\rho}{\psi(r)^{n-1}} \, dr = \infty \, .
\end{equation}
In particular, we can consider the relevant case (let $C>0$)
\begin{equation}\label{quad}
\mathrm{Ric}_o(x) \ge - C \left[1 + r(x)^2 \right] \qquad  \forall x \in M \setminus \left(  \{ o \} \cup \operatorname{cut}(o)  \right)  .
\end{equation}

The fast diffusion equation has widely been investigated in the Euclidean setting, see e.g.~\cite{V} and references quoted therein for a thorough discussion. In that setting it takes origin as a model for plasma physics \cite[Chapter 2]{V}, it comes into play in the diffusive limit of kinetic equations \cite{LT} and, for a special value of $m$, in the evolutionary Yamabe problem (see for instance \cite{DDS}). Note that the FDE is a \it singular \rm equation in the sense that the diffusion coefficient $m|u|^{m-1}$ diverges as $|u|\to0$, but at the same time it is also \emph{degenerate} since the diffusion coefficient vanishes as $ |u| \to \infty $. Solutions to \eqref{cauchy} exhibit infinite speed of propagation in $ \mathbb{R}^n $, to such an extent that finite-time extinction can occur if $m$ is sufficiently close to zero.

The investigation about \it nonlinear diffusions \rm of the type of \eqref{cauchy} on Riemannian manifolds has begun just recently, dealing especially with the case $m>1$, known in the literature as \it porous medium equation \rm (PME) or \emph{slow diffusion}. We mention in this connection the papers \cite{V2, GM, GMV, BS, GMP, GMP2, GMV-MA}, where several well-posedness issues have successfully been addressed. Nevertheless, it should be pointed out that the methods used in the PME regime are often very different from the ones that are suitable for the analysis of the FDE, and the results themselves exhibit significant dissimilarities: for example, the natural class of initial data for which existence and uniqueness of solutions to the FDE is guaranteed in ${\mathbb R}^n$ is considerably \it larger \rm than the one corresponding to the PME.

The Cauchy problem \eqref{cauchy} in the case $M \equiv {\mathbb R}^n$ and $ m \in (0,1) $ was thoroughly investigated in the seminal paper \cite{HP} by Herrero and Pierre. They show that, for any $L^1_{\mathrm{loc}}(\mathbb{R}^n)$ initial datum, there exists a global solution, which is also unique under the additional assumption that $u_t$ is locally integrable, i.e.~in the class of \emph{strong solutions}.  In particular, no requirement at all on the behavior at infinity of $ u_0 $ is necessary. 

A crucial tool in their proofs is the celebrated \emph{Herrero-Pierre estimate}, that allows one to bound the $L^1$ norm of the solution (at some time) in a ball by means of the $L^1$ norm of the solution (at a different time) in a larger ball, plus an explicit remainder term. More importantly, the same estimate applies to the \emph{difference} of solutions, see Propositions \ref{prop1} and \ref{prop2} below for analogues in the present setting. We mention that uniqueness of purely \emph{distributional} solutions was not addressed in \cite{HP}, and to our knowledge it had not yet been studied so far even in the Euclidean space. Nevertheless, we recall that uniqueness of nonnegative \emph{bounded} distributional solutions for problem \eqref{cauchy} ($M\equiv\mathbb R^n$) with the additional absorption term $-u^p$ in the right-hand side, for $u_0\in L^\infty(\mathbb R^n)$, was achieved in \cite[Theorem 2.1]{PZ} provided $m>(1-2/ n )^+ $ and $ p> n (1-m)/2$.

As concerns the Riemannian setting, a first contribution was given by \cite{BGV}, where the FDE is investigated in the special class of Cartan-Hadamard manifolds, namely complete, simply connected Riemannian manifolds with everywhere nonpositive sectional curvature. In particular, problem \eqref{cauchy} is shown to admit a unique strong solution, in the $ H^{-1}(M) $ sense, if $u_0 $ is \emph{globally} integrable and belongs to $ H^{-1}(M) $ (see Theorems 3.2 and 3.4 there for the details).

In the recent paper \cite{BS}, uniqueness of strong solutions to \eqref{cauchy} for \it locally \rm integrable data is proved on general manifolds, following the Herrero-Pierre strategy, by means of a careful construction of appropriate cut-off functions, under the additional requirement that the (negative) curvature decays sufficiently fast at infinity. More precisely, it is assumed that
\begin{equation}\label{quasi-euc}
\mathrm{Ric}(x) \geq -  \frac{C}{1+r(x)^2} \qquad  \forall x \in M
\end{equation}
for some $ C > 0$ and $m>m_c $, where $m_c$ is a suitable exponent depending on $C$ and $n$, larger than $1-2/ n  $. If $ C=0 $, the result holds for every $ m \in (0,1) $ (see \cite[Theorem 4.9]{BS}).

In \cite{GIM}, uniqueness of (nonnegative) \emph{globally bounded} solutions for the FDE is proved to be fully equivalent to the \it stochastic completeness \rm of the manifold, a property that does hold under \eqref{eq1b}--\eqref{eq1c} \cite[Corollary 15.3]{Gri}. {Furthermore, on model manifolds, namely spherically-symmetric Riemannian manifolds where \eqref{eq1b} is in fact an identity, condition \eqref{eq1c} turns out to be equivalent to stochastic completeness \cite[Proposition 3.2]{Gri}.} In particular, {if one considers power-type curvature bounds}, uniqueness of (nonnegative, bounded) solutions to the FDE is ensured under \eqref{quad}, but it \it fails \rm as soon as $\mathrm{Sec}_o(x) \le - C \, r(x)^{2+\varepsilon}$ for some $C,\varepsilon>0$ and large $r(x)$, where $\mathrm{Sec}_o(x)$ stands for the radial sectional curvature {(here $o$ is required to be a \emph{pole}, namely to have empty cut locus)}.

We finally mention \cite{GM1}, where fine long-time asymptotics for solutions to \eqref{cauchy} corresponding to a restricted class of (radial) initial data is investigated in the special case of the hyperbolic space $ \mathbb{H}^n $.

In the present paper, we focus on the main open well-posedness issues related to the Cauchy problem \eqref{cauchy}. We will prove the following results:
\begin{itemize}
\item Theorems \ref{strong-sol} and \ref{distr-sol}: existence of \it distributional solutions \rm 
for general data $ u_0 \in L^1_{\mathrm{loc}}(M) $, and of the \it minimal solution \rm for nonnegative data $ u_0 \in L^2_{\mathrm{loc}}(M)$, {without curvature bounds}.

\item Theorem \ref{tuni1}: uniqueness of \it strong solutions \rm under the curvature bounds \eqref{eq1b}--\eqref{eq1c}, provided they have a common $L^1_{\mathrm{loc}}(M)$ initial trace in a suitable sense.

    \item Theorem \ref{tuni2}: uniqueness of \it nonnegative distributional solutions \rm under the curvature bounds \eqref{eq1b}--\eqref{eq1c}, provided $u\in L^2_{\mathrm{loc}}(M\times[0,+\infty))$. In fact, we will show that any such solution must coincide with the  minimal one constructed in Theorem \ref{distr-sol}. Apparently, this property was not known even in $ \mathbb{R}^n $.
\end{itemize}

{In the light of the nonuniqueness results of \cite{GIM}, the curvature conditions appearing in Theorems \ref{tuni1} and \ref{tuni2} are sharp. Indeed, if $ o $ is a pole and $\mathrm{Sec}_o(x) \le - \psi''(r(x))/\psi(r(x)) $ for some function $\psi$ as above that \emph{does not} meet \eqref{eq1c}, then $M$ is stochastically incomplete \cite[Corollary 15.3]{Gri}. Hence \cite[Theorem 1.1]{GIM} implies that any nonnegative initial datum $ u_0 \in L^\infty(M) $ gives rise to at least two different nonnegative bounded solutions to \eqref{cauchy}. As mentioned above, this occurs in particular on manifolds with \emph{superquadratic} negative curvature.}

In order to prove our uniqueness results for \eqref{cauchy} it will be crucial to deal,  for every $p>1$ and $\alpha>0$, with the \emph{semilinear elliptic equation}
\begin{equation}\label{eqn1}
\Delta W = \alpha \, W \, |W|^{p-1} \qquad \text{in } M \, .
\end{equation}
More precisely, we will establish the following nonexistence property:

\begin{itemize}

\item Theorem \ref{tell}: for every $p>1$ and $ \alpha>0$ equation \eqref{eqn1} does not admit any nonnegative, nontrivial, distributional \emph{subsolution}, namely any nonnegative function $ W \not \equiv 0 $ with $W\in L^p_{\mathrm{loc}}(M)$ such that
\begin{equation}\label{eqn1-subsol}
\Delta W \ge  \alpha \, W^{p} \qquad \text{in } \mathcal{D}'(M) \, .
\end{equation}
\end{itemize}
{In other words, should such a function satisfy \eqref{eqn1-subsol}, then it is identically zero. As an immediate corollary, in view of Kato's inequality one can deduce that neither \eqref{eqn1} admits (possibly sign-changing) distributional solutions.}

\begin{oss}\rm
{Results on nonexistence of nontrivial solutions for elliptic problems of the form \eqref{eqn1} or \eqref{eqn1-subsol}} have widely been studied in the literature, even for more general operators and nonlinearities. In particular, similar properties are known to hold without requiring curvature bounds, but assuming an \emph{a-priori} growth condition on the solution instead (see \cite{PRS} and references therein). On the other hand, when a suitable curvature bound from below is satisfied, it was shown in \cite{M} and \cite[Theorem 1.9]{PRS} that the so-called \emph{Omori-Yau} maximum principle is valid. The corresponding curvature condition does single out quadratic growth, {however it can be checked that it is slightly stronger than \eqref{eq1b}--\eqref{eq1c}, since the non-integrability constraint (i.e.~the analogue of \eqref{eq1c}) actually involves $ \psi'' $.} Nevertheless, if such a maximum principle holds \cite[Theorem 1.31]{PRS} establishes first that all $C^2(M)$ nonnegative subsolutions to \eqref{eqn1} are bounded, and as a consequence must vanish identically {by previous results}. Therefore, {up to a small gap in the curvature bounds}, the thesis of Theorem \ref{tell} was already known for \emph{smooth subsolutions}. Our main contribution consists of dropping the regularity assumption, {since we deal with merely distributional subsolutions}. Clearly distributional \it solutions \rm to \eqref{eqn1} are {at least $C^2(M)$} by elliptic regularity, but it should be commented that our results do not depend in any way on regularity {(we only exploit local boundedness)} and hence are potentially applicable to more general contexts. {It should however be pointed out that in \cite{PRS} the right-hand side of \eqref{eqn1-subsol} includes a wider class of nonlinearities.}
\end{oss}

We stress that, under assumptions \eqref{eq1b}--\eqref{eq1c}, we are able to extend the results of \cite{GIM} so as to cover the whole class of initial data and solutions from $ L^\infty $ to {$L^2_{\mathrm{loc}}$}. Moreover, note that the methods of proofs used in \cite{HP} and in \cite{PZ} to obtain uniqueness do not work in our framework. To be specific, the arguments exploited in \cite{HP} require a sort of homogeneity which is typical of $\mathbb R^n$, and is in general lost on Riemannian manifolds. Furthermore, the authors take advantage of the classical mean-value property for (sub-) harmonic functions; the latter remains true also on general Riemannian manifolds, however it has a different local form (see \cite{LS}) which makes it unsuitable for a straight application of the methods of proof developed in \cite{HP}. On the other hand, the main argument of \cite{PZ} (and also of \cite[Theorem 4.9]{BS} under \eqref{quasi-euc}) relies on the fact that the volume of balls has polynomial growth, whereas in our setting, in view of the possible negative curvature, geodesic balls can even grow exponentially or faster. Hence, although we take inspiration from various ideas of \cite{HP} and \cite{PZ}, here we exploit a different strategy to achieve uniqueness. In this regard, some techniques will also be borrowed from the ``bounded'' framework of \cite{GIM}.

\subsection{Statements of the main results}\label{state}

We start by providing the definitions of solution that we will deal with. In the following, we let $ d\mu $ denote the Riemannian volume measure of the manifold at hand $M$, which hereafter is assumed to be complete, connected and noncompact (unless otherwise specified).

\begin{definition}[Distributional solutions]\label{defsol}
Let $m \in (0,1)$ and $ u_0 \in L^1_{\mathrm{loc}}(M) $. We say that a function $u\in L^1_{\mathrm{loc}}(M\times [0, +\infty))$ is a (distributional, or very weak) solution of the Cauchy problem \eqref{cauchy} 
if it satisfies
\begin{equation}\label{distrib-id}
u_t=\Delta u^m \qquad \text{in } \mathcal D'(M\times (0,+\infty))
\end{equation}
and (in the sense of essential limits)
\begin{equation}\label{eq12c}
\lim_{t\to 0^+} \int_M u(x,t) \, \psi(x) \, {d}\mu  = \int_M u_0 \, \psi  \, d\mu \qquad \forall \psi\in C^\infty_c(M)\, .
\end{equation}
\end{definition}

Our main existence result, without further assumptions on $ u_0 $, is the following.
\begin{theorem}[Existence of solutions]\label{strong-sol} There exists a solution $ u $ of problem \eqref{cauchy}, in the sense of Definition \ref{defsol}. In addition $u\in C([0, +\infty); L^1_{\mathrm{loc}}(M))$.
\end{theorem}

For \emph{nonnegative} initial data that also belong to $L^2_{\mathrm{loc}}(M)$, we can establish existence of the \emph{minimal} solution in the class of nonnegative distributional solutions.

\begin{theorem}[Existence of the minimal solution]\label{distr-sol}
Given $ u_0 \in L^2_{\mathrm{loc}}(M)$, with $u_0\geq 0 $, there exists a nonnegative {distributional} solution $\underline u\in L^2_{\mathrm{loc}}(M\times[0,+\infty))$ of problem \eqref{cauchy}, in the sense of Definition \ref{defsol}, which is {minimal} in the class of nonnegative distributional solutions belonging to $L^2_{\mathrm{loc}}(M\times[0,+\infty))$. That is, for any nonnegative distributional solution $u\in L^2_{\mathrm{loc}}(M\times[0,+\infty))$ of \eqref{distrib-id} with the same initial datum according to \eqref{eq12c} we have
$$\underline u\leq u \qquad \text{a.e.~in }  M \times (0, +\infty) \, . $$
\end{theorem}

In some of our results below we will need to require a further property of solutions, which amounts to asking that the time derivative is a locally integrable function, so that \eqref{distrib-id} holds pointwise.

\begin{definition}[Strong solutions]\label{strong} We say that a function $u$ is a {strong solution} of problem \eqref{cauchy} if it is a distributional solution in the sense of Definition \ref{defsol} and in addition
\begin{equation*}\label{eq11}
u_t \in L^1_{\mathrm{loc}}(M\times (0, +\infty)) \, .
\end{equation*}
\end{definition}

If moreover the Ricci curvature complies with conditions \eqref{eq1b}--\eqref{eq1c}, we are able to obtain a uniqueness result for {\it strong solutions}. We stress that, even in the Euclidean setting, uniqueness of possibly \emph{sign-changing} solutions was proved within the class of strong solutions only (see \cite[Theorem 2.3]{HP}).

\begin{theorem}[Uniqueness of strong solutions]\label{tuni1}
Let the curvature conditions \eqref{eq1b}--\eqref{eq1c} be satisfied. Let $u $ and $ v $ be any two strong solutions of problem \eqref{cauchy}, in the sense of Definition \ref{strong}, such that $\left|u(\cdot,t)-v(\cdot,t)\right|\to 0$ in $L^1_{\mathrm{loc}}(M)$ as $t\to0^+$. Then $u= v$ almost everywhere in $M\times (0, +\infty)$.
\end{theorem}

{Note that in the above result the initial condition \eqref{eq12c} is irrelevant. The sole important requirement is that the difference between $ u  $ and $ v $ vanishes in $ L^1_{\mathrm{loc}}(M) $ as $ t \to 0^+ $. Clearly this is the case for strong solutions taking the same initial datum $ u_0 \in L^1_{\mathrm{loc}}(M)$ which also belong to $C([0, +\infty); L^1_{\mathrm{loc}}(M))$.}

For nonnegative initial data belonging to $L^2_{\mathrm{loc}}(M)$, we do not need solutions to be strong in order to establish uniqueness. Indeed, we can show that any nonnegative distributional solution must coincide with the minimal one, constructed in Theorem \ref{distr-sol}. To the best of our knowledge this is new also in the case $M \equiv \mathbb R^n$, since in the literature uniqueness is typically proved for {strong} solutions (see again \cite[Theorem 2.3 and remarks below]{HP}) {or for bounded solutions \cite{PZ,GIM}} only.

\begin{theorem}[Uniqueness of nonnegative distributional solutions]\label{tuni2}
Let the curvature conditions \eqref{eq1b}--\eqref{eq1c} be satisfied and $u_0\in L^2_{{\mathrm{loc}}}(M)$, with $u_0\ge0$. Let $u\in L^2_{\mathrm{loc}}(M\times[0,+\infty))$ be a nonnegative distributional solution of problem \eqref{cauchy}, in the sense of Definition \ref{defsol}. Then $u = \underline u$ almost everywhere in $M\times (0, +\infty)$, where $\underline u$ is the minimal solution provided by Theorem \ref{distr-sol}.
\end{theorem}

Both our uniqueness results rely on a crucial nonexistence theorem for the nonlinear elliptic equation \eqref{eqn1}, in the spirit of Keller and Osserman \cite{K,O}. Let us emphasize that the latter is of independent interest and the proof we will provide is self-contained and does not exploit methods of parabolic equations.

\begin{theorem}[Nonexistence for the elliptic equation]\label{tell} Let $p>1$, $ \alpha>0 $ and the curvature conditions \eqref{eq1b}--\eqref{eq1c} be satisfied. Then:

\noindent (i) there exists no nonnegative, {nontrivial}, distributional subsolution to \eqref{eqn1};

\noindent (ii) there exists no {nontrivial} distributional solution of \eqref{eqn1}.

\end{theorem}

\subsection{Plan of the paper} The paper is organized as follows. In Section \ref{basics} we provide a concise adaptation to the present setting of the Herrero-Pierre estimates established in \cite{HP}, that will allow us to prove Theorem \ref{strong-sol}. Section \ref{sect-minimal} contains some technical but key tools involving an integration-by-parts formula for nonsmooth (distributional) supersolutions to suitable parabolic equations that are relevant to our purposes, which will first be qualitatively discussed in Subsection \ref{intro} and then precisely addressed in Subsection \ref{ipp}. After the construction of a candidate minimal solution for nonnegative initial data in Subsection \ref{constr-min}, such tools will be exploited in Subsection \ref{local-comp} to show a local comparison principle, which will permit us to infer that the constructed solution is indeed minimal, namely Theorem \ref{distr-sol}. The nonexistence results for the elliptic problems \eqref{eqn1} and \eqref{eqn1-subsol}, i.e.~Theorem \ref{tell}, are proved in Section \ref{Ellptic}. The latter will be employed in Section \ref{uniq} to carry out the proofs of Theorems \ref{tuni1} and \ref{tuni2}.

\section{Herrero-Pierre estimates and existence of general solutions}\label{basics}

The aim of this section is to first introduce some basics of {exhaustion functions}, which on general manifolds allow one to replace balls (that may not be regular enough), and then establish some key local estimates for (approximate) solutions to \eqref{cauchy} that will be crucial in order to prove Theorems \ref{strong-sol} and \ref{distr-sol}.

\subsection{Exhaustion function and regularized distance}\label{regdist}
By well-known results, for which we refer e.g.~to \cite[Propositions 2.28, 5.47 and Theorem 6.10]{Lee}, on any connected, noncompact, $ n $-dimensional Riemannian manifold $ M $ there exists a regular \emph{exhaustion function}, namely a (surjective) smooth function $ \mathcal{E} : M \to [0,\infty) $ having the following property: for almost every $ R \in (0,\infty) $ the sublevel set
\begin{equation}\label{sub}
{\Omega}_R := \left\{ x \in M : \ \mathcal{E}(x) < R \right\}
\end{equation}
is a regular precompact domain of $M$, i.e.~$\overline{\Omega}_R $ is an $n$-dimensional compact submanifold with boundary which is properly embedded in $ M $. By construction, the boundary of  $ \Omega_R $ equals the level set $ \left\{ x \in M : \ \mathcal{E}(x) = R \right\} $, and its outward-pointing normal field with respect to $ \Omega_{R} $ is provided by $ x \mapsto \nabla \mathcal{E}(x) / \left| \nabla \mathcal{E}(x) \right| $.
We call all such $R$ \emph{regular value} for $ \mathcal{E} $, having in mind Sard's Theorem \cite[Theorem 6.10]{Lee}. The term ``{exhaustion}'' comes from the fact that, by the definition, for every strictly increasing sequence $ R_{k} \to \infty $ we have
$$
\overline{\Omega}_{R_k} \Subset \Omega_{R_{k+1}}  \qquad \text{and} \qquad \bigcup\limits_{k=1}^{\infty} \Omega_{R_k} = M \, .
$$
More in general, we say that a sequence of open sets $ \{ D_k \} \subset M $ is a \emph{regular exhaustion} of $ M $ if each $  {D}_k $ is a regular precompact domain and it satisfies
$$
\overline{D}_k \Subset D_{k+1}  \qquad \text{and} \qquad \bigcup\limits_{k=1}^{\infty} D_{k} = M \, .
$$
In fact, {if $M$ is complete}, we can assume that for every $\varepsilon>0$ there exists an exhaustion function $\mathcal{E}_{\varepsilon}$ as above which in addition fulfills
\begin{equation}\label{lev}
\left| \mathcal{E}_{\varepsilon}(x) - r(x) \right| < \varepsilon \qquad \forall x \in M \, ,
\end{equation}
where $ M \ni x \mapsto r(x) := \operatorname{d}(x,o) $ stands for the geodesic distance from a (fixed) reference point $ o \in M $. This is a consequence of \cite[Proposition 2.1]{GW}, since $ r(x)$ is a $1$-Lipschitz function. For every $R>0$, let $\Omega_{R, \varepsilon}$ be the sublevel set
\begin{equation}\label{neq20}
\Omega_{R, \varepsilon}:= \left\{x \in M : \ \mathcal{E}_{\varepsilon}(x) < R \right\}  .
\end{equation}
In particular, by virtue of \eqref{lev} we deduce that
\begin{equation}\label{neq8}
B_{R-\varepsilon}(o) \subset \Omega_{R, \varepsilon} \subset B_{R+\varepsilon}(o) \qquad \forall R> \varepsilon \, ,
\end{equation}
where $ B_\rho(o) $ denotes the geodesic ball of radius $ \rho>0 $ centered at $ o $. Sard's Theorem (we refer again to \cite[Proposition 5.47 and Theorem 6.10]{Lee}) guarantees that there exists a negligible set $\mathcal N_{\varepsilon}  \subset (0,\infty) $ such that for every $ R\in (0, \infty)\setminus \mathcal N_{\varepsilon}$ the sublevel set $ \Omega_{R, \varepsilon} $ is indeed a regular precompact domain of $M$.

\subsection{Existence proof through Herrero-Pierre estimates}\label{HP}

In order to establish the analogues of the Euclidean Herrero-Pierre estimates (we refer in particular to \cite[Lemma 3.1]{HP} and the beginning of the proof of \cite[Theorem 2.3]{HP}), first of all it is important to provide a suitable family of regular \emph{cut-off} functions. This can easily be done as follows. Fix any $\varepsilon_0 > 0$ and set $\mathcal E\equiv \mathcal E_{\varepsilon_0}$ and $ \Omega_{R} \equiv \Omega_{R, \varepsilon_0} $ for every $ R > 0 $. Let $\phi\in C^\infty([0, \infty))$ be any function satisfying
$$
0 \leq \phi  \leq 1 \, , \qquad \phi = 1 \quad \text{in } [0,1] \, , \qquad \phi=0 \quad \text{in } [2, \infty) \, .
$$
For every $R>0$, we put
\begin{equation}\label{eqn2}
\phi_R(x):= \phi\!\left(\frac{\mathcal E(x)}{R}\right) \qquad \forall x\in M\, .
\end{equation}

We are now in position to prove the claimed estimates of Herrero-Pierre type on the general class of manifolds we deal with. Since the techniques employed are basically the same as in \cite{HP}, only a concise argument will be provided for the reader's convenience.

\begin{proposition}
	\label{prop1}
Let $ m \in (0,1) $. Let $u, v\in L^1_{\mathrm{loc}}(M\times (0,+ \infty))$ with $u\geq v$. Suppose that
\begin{equation}\label{eq1}
u_t=\Delta u^m \quad \text{and} \quad v_t=\Delta v^m \qquad \text{in } \mathcal D'(M\times (0,+\infty)) \, .
\end{equation}
Let $ R > 0$. Then the following estimate holds:
\begin{equation}\label{eq3}
\left[\int_{\Omega_R}\left[u(x,t)-v(x,t)\right] d\mu \right]^{1-m} \leq \left[\int_{\Omega_{2 R}}\left[u(x,s)-v(x,s)\right] d\mu \right]^{1-m} + \mathcal{H}_{R} \left|t-s\right|
\end{equation}
for almost every $t,s\in (0, +\infty)$, where
\begin{equation}\label{eq4}
\mathcal H_{R} := \kappa_m \sup_{\Omega_{2R}\setminus \Omega_R }\left\{\left|\nabla \phi_R\right|^2+\left|\Delta \phi_R\right|\right\} \left[ \mu\!\left( \Omega_{2R}\setminus \Omega_R \right) \right]^{1-m}  ,
\end{equation}
the constant $\kappa_m>0$ depending only on $m$.
\end{proposition}
\begin{proof}
By the proceeding along the lines of the proof of \cite[Proposition 7.3]{BGV}, which in turn relies on \cite[Lemma 3.1]{HP}, one can show the validity of the following inequality:
\begin{equation}\label{eq5}
\begin{aligned}
\left[\int_{M}\left[u(x,t)-v(x,t)\right] \psi(x) \, d\mu \right]^{1-m} \leq & \left[ \int_{M}\left[u(x,s)-v(x,s)\right] \psi(x) \, d\mu \right]^{1-m} \\ & + (1-m) \, C(\psi) \, |t-s|
\end{aligned}
\end{equation}
for almost every $t,s \in (0, +\infty)$, for any nonnegative $\psi\in C^\infty_c(M)$, where
\begin{equation}\label{eq6}
C(\psi) := 2^{1-m} \left[ \int_M \left| \Delta \psi \right|^{\frac 1{1-m}} \psi^{-\frac m{1-m}} \, d\mu \right]^{1-m} .
\end{equation}
Note that, in contrast to \cite{BGV} and \cite{HP}, here we do not ask solutions to be continuous curves in $L^1_{\mathrm{loc}}(M)$; this is not an issue, since such a requirement was used only to ensure that \eqref{eq5} (and subsequent estimates) holds at {every} $ t,s \ge 0 $, whereas in this case it is enough to consider Lebesgue points of $ u $ and $v$ as curves in $ L^1_{\mathrm{loc}}((0,+ \infty);L^1_{\mathrm{loc}}(M)) $. In order to bound the quantity $C(\psi)$, we make an appropriate choice of $ \psi $. That is, let us pick $\psi \equiv \phi_R^{k}$ for any integer $k \ge {2}/{(1-m)}$, where $\phi_R$ is precisely the cut-off function given in \eqref{eqn2}. We have:
\begin{equation}\label{eq7}
\begin{aligned}
\left|\Delta \psi \right|^{\frac 1{1-m}} \psi^{-\frac m{1-m}} =& \left|k(k-1)\,\phi_R^{k-2}\left|\nabla \phi_R\right|^2+k \, \phi_R^{k-1} \Delta\phi_R\right|^{\frac{1}{1-m}} \phi_R^{-\frac{k m}{1-m}}\\
\leq & \left[k(k-1)\right]^{\frac 1{1-m}} \phi_R^{\frac{k(1-m)-2}{1-m}} \left( \left|\nabla\phi_R\right|^2 + \left|\Delta\phi_R\right| \right)^{\frac1{1-m}} .
\end{aligned}
\end{equation}
In view of \eqref{eq6}, \eqref{eq7}, the support properties of $\phi_R$ and the fact that $ \phi_R \le 1 $, we can therefore infer that
\begin{equation}\label{eq8}
\begin{aligned}
C(\psi)  = C(\phi_R) = & \, 2^{1-m} \left[ \int_M \left| \Delta \psi \right|^{\frac 1{1-m}} \psi^{-\frac m{1-m}} \, d\mu \right]^{1-m} \\
\leq & \, 2^{1-m} \, k(k-1) \sup_{\Omega_{2R}\setminus \Omega_R }\left\{\left|\nabla \phi_R\right|^2+\left|\Delta \phi_R\right|\right\} \left[ \int_{\Omega_{2R}\setminus \Omega_R} d\mu \right]^{1-m} .
\end{aligned}
\end{equation}
Hence, from \eqref{eq5} with $ \psi \equiv \phi_R $ (exploiting again the support properties of $ \phi_R $) and \eqref{eq8},  the thesis follows.
\end{proof}

An analogue of Proposition \ref{prop1} can be shown without assuming $u\geq v$, provided solutions are strong; see the first part of the proof of \cite[Theorem 2.3]{HP} in the Euclidean space, in particular formula (3.32) there.

\begin{proposition}
	\label{prop2}
	Let $ m \in (0,1) $ and $u, v \in L^1_{\mathrm{loc}}(M\times (0,+ \infty))$ satisfy \eqref{eq1}. Suppose in addition that $ u_t , v_t \in L^1_{\mathrm{loc}}(M\times (0,+ \infty)) $. Let $ R>0 $. Then the following estimate holds:
\begin{equation*}\label{eq3bis}
\left[\int_{\Omega_R}\left|u(x,t)-v(x,t)\right| {d}\mu \right]^{1-m} \leq \left[\int_{\Omega_{2 R}} \left|u(x,s)-v(x,s)\right| {d}\mu \right]^{1-m} + \mathcal{H}_{R} \, |t-s|
\end{equation*}
for every $ t,s \in (0,+\infty) $, where the constant $\mathcal H_{R}$ is the same as in \eqref{eq4}.
\end{proposition}
\begin{proof}
Since $ u_t $ and $ v_t $, thus $ \Delta u^m $ and  $ \Delta v^m $, are locally integrable functions, we can apply Kato's inequality \cite[Lemma A]{Kato} to infer that
\[
-\Delta \left|u^{m} -v^{m} \right| \leq -\operatorname{sign}(u-v) \, \Delta \left(u^{m}- v^{m}\right) \qquad \text{in } \mathcal D'(M\times (0, +\infty))  \, .
\]
Thus, using \eqref{eq1}, we obtain:
\begin{equation}\label{eq14}
\frac{\partial}{\partial t}\left|u-v\right| \le\Delta\left| u^{m} - v^{m} \right| \qquad \text{in } \mathcal D'(M\times (0, +\infty)) \, .
\end{equation}
As a consequence, for any nonnegative $\psi\in C^\infty_c(M)$ we have:
$$
\begin{aligned}
\frac{d}{dt} \int_M \left|u(x,t)-v(x,t) \right| \psi(x) \, d\mu  & \leq \int_M \left| u(x,t)^{m} - v(x,t)^{m} \right| \left|\Delta \psi(x) \right|  {d} \mu \\
&\leq C(\psi) \left(\int_M  \left|u(x,t)-v(x,t)\right| \psi(x) \, {d} \mu  \right)^m ,
\end{aligned}
$$
where $C(\psi)$ is defined by \eqref{eq6}. This implies the validity of \eqref{eq5} (with moduli inside the integrals), whence the thesis follows by arguing exactly as in the proof of Proposition \ref{prop1}.
\end{proof}

\begin{proof}[Proof of Theorem \ref{strong-sol}]
One can exploit the same strategy as in \cite[Theorem 2.1]{HP}, upon taking advantage of Proposition \ref{prop1}. {We only mention that the basic idea consists of solving problem \eqref{cauchy} starting from approximate initial data $ 	\{ u_{0k} \}  \subset L^1(M) \cap L^\infty(M) $, such that $ \lim_{k \to \infty} u_{0k} = u_0  $ in $ L^1_{\mathrm{loc}}(M) $, for which existence and uniqueness theory is well established (no significant difference occurs compared to the case $ M \equiv \mathbb{R}^n $). The Herrero-Pierre estimates \eqref{eq3} applied to the corresponding sequence of solutions $ \{ u_k \} $, which enjoys monotonicity properties with respect to $k$ if the initial data are chosen properly, play a key role in the passage to the limit as $ k \to \infty $ since they provide $ L^1_{\mathrm{loc}}(M) $ stability.}
\end{proof}

\section{Nonnegative distributional solutions and existence of the minimal one}\label{sect-minimal}

The existence proof of the minimal solution, in the class of nonnegative distributional solutions, requires first the validity of a \emph{local comparison} principle. To this aim, we need to develop some technical tools involving an integration-by-parts formula for \emph{nonsmooth} supersolutions to certain elliptic/parabolic problems.

\subsection{Integration by parts: introduction to the problem}\label{intro}
Before stating the rigorous formula, see Proposition \ref{lem:comp} below, we briefly recall the classical counterpart in the smooth setting and explain its connection with the nonsmooth one. Let $ f \in C^2(M)$ and $ g \in C(M) $. Suppose that we have
\begin{equation}\label{classical-1}
\Delta f \le g \qquad \text{in } M \, .
\end{equation}
Then, for every $ R>0 $ and every nonnegative test function $ \eta \in C^2\!\left(\overline{B}_R(o)\right) $ with $ \eta = 0 $ on $ \partial B_R(o) $, by \eqref{classical-1} and the divergence theorem we obtain:
\begin{equation}\label{classical-2}
\int_{B_R(o)} f \, \Delta \eta \, d\mu - \int_{\partial B_R(o)} f \, \frac{\partial \eta}{\partial \nu} \, d\sigma \le \int_{B_R(o)} g \, \eta \, d\mu =: F(\eta) \, ,
\end{equation}
where $ d\sigma $ stands for the $(n-1)$-dimensional Hausdorff measure on $  \partial B_R(o) $, $ \nu $ is the outward normal direction and we are implicitly assuming that balls are regular sets. In general, we can interpret $ F $ as a functional on $ C_c(M) $, upon extending $ \eta $ to zero outside $ \overline{B}_R(o)  $. Furthermore, on the one hand the first addendum in the left-hand side of \eqref{classical-2} is well defined provided $ f $ merely belongs to $ L^1_{\mathrm{loc}}(M)$; on the other hand, it is not difficult to show that for locally integrable functions the second addendum is also well defined at least for \emph{almost every} $R>0$. Thus, the following natural question arises: can \eqref{classical-2} be proved only upon assuming that $f\in L^1_{\mathrm{loc}}(M)$ and \eqref{classical-1} holds in $ \mathcal{D}^\prime(M) $ with right-hand side in the dual of $ C_c(M) $, i.e.
$$
\int_{M} f\, \Delta \xi \, d\mu \le F(\xi)
$$
for every nonnegative $ \xi \in C^\infty_c(M) $, where $F $ is a continuous functional on $C_c(M)$? We will now show that the answer is indeed positive, provided one replaces $B_R(o)$, which on general manifolds need not be more than Lipschitz regular, by the sublevel sets of a suitable \emph{exhaustion function} of $M$ (recall Subsection \ref{regdist}). This is basically the content of the next Proposition \ref{lem:comp}. Note that in formulas \eqref{p1}--\eqref{p2} below $ f $ and the test functions also depend on time, so that the functional is actually defined on $ C_c(M \times [0,T]) \times  C_c(M \times [0,T]) $. Nevertheless, this is irrelevant to the above discussion.

\subsection{An integration-by-parts formula for merely integrable functions}\label{ipp}

We are ready to state and prove the key result of this section, namely a generalized version of \eqref{classical-2}.

\begin{proposition}\label{lem:comp}
Let $ M $ be a connected, noncompact Riemannian manifold of dimension $n$. Let $ T>0 $ and $ f \in L^1_{\mathrm{loc}}(M \times [0,T]) $ satisfy
\begin{equation}\label{p1}
\int_0^T \int_{M} f \, \Delta \xi \, d\mu dt \le F(\xi,\xi_t) \qquad \forall \xi \in C^2_c(M \times [0,T]): \ \, \xi \ge 0 \, ,
\end{equation}
where $ F $ is a continuous functional on $ C_c(M \times [0,T]) \times  C_c(M \times [0,T]) $. Let $ \mathcal{E} $ be an exhaustion function of $M$ and  $ R_k \to \infty $ be any strictly increasing sequence of regular values for $ \mathcal{E} $. Then there exists a regular exhaustion $ \{ D_k \} \subset  M $, possibly depending on $f$, such that
\begin{equation}\label{double-inclusion}
\overline{\Omega}_{R_k} \Subset D_k \subset \overline{D}_k \Subset \Omega_{R_{k+1}}
\end{equation}
and
\begin{equation}\label{p2}
\begin{gathered}
\int_0^T \int_{D_k} f \, \Delta \eta \, d\mu dt - \int_0^T \int_{\partial D_k} f \, \frac{\partial \eta}{\partial \nu} \, d\sigma dt \le F\!\left(\overline{\eta},\overline{\eta}_t\right) \\
\forall \eta \in C^2\!\left(\overline{D}_k \times [0,T]\right)\!: \ \, \left.\eta\right|_{\partial D_k \times [0,T] }=0 \, , \ \, \eta \ge 0 \, ,
\end{gathered}
\end{equation}
for every $ k \in \mathbb{N} $, where $ \nu $ is the outward-pointing normal field on $ \partial D_k $, $ d\sigma $ is the $(n-1)$-dimensional Hausdorff measure on $ \partial D_k $ and $ \overline{\eta} $ is the extension of $ \eta $ to $ M \times [0,T] $, set to zero outside $ \overline{D}_k \times [0,T] $.
\end{proposition}
\begin{proof}
Given $ \epsilon>0  $ and any regular value $ R $ for $ \mathcal{E} $, let $ \Omega^\epsilon_R $ denote the set of all points of $ \Omega_R $ (defined in \eqref{sub}) whose geodesic distance from $ \partial \Omega_R $ is less than $ \epsilon $:
\begin{equation*}\label{sigma1}
\Omega^\epsilon_R := \left\{ x \in \Omega_R: \ \operatorname{d}(x,\partial\Omega_R) < \epsilon \right\} .
\end{equation*}
Because $ \partial\Omega_R $ is a smooth $(n-1)$-dimensional submanifold of $ M $, we can and will suppose that $ \epsilon $ is so small that the projection $ \pi(x) $ of a point $ x \in \Omega^\epsilon_R $ onto $ \partial \Omega_R $ is single valued, regular and the map $ \Pi $ that with each $ x \in \Omega^\epsilon_R $ associates the pair $ (\pi(x),\mathrm{d}(x,\pi(x))) $ is a diffeomorphism between $ \Omega^\epsilon_R $ and $ \partial \Omega_R \times (0,\epsilon) $. In this way, one can completely describe $ \Omega^\epsilon_R $ by means of the coordinates $ (y,\delta) $, as $ y $ ranges on $ \partial \Omega_R $ and $ \delta $ ranges in $ (0,\epsilon) $. We refer e.g.~to \cite{Foote} for more details on such a local construction (the fact that the setting there is Euclidean is inessential).

Let $ \Sigma := \partial \Omega_R  $. The integrability properties of $f $ yield $ f \in L^1(\Omega_R^\epsilon \times (0,T) ) $, which is equivalent to claiming that the function $\hat{f} : \Sigma \times (0,\epsilon) \times (0,T) \to \mathbb{R} $ defined by
$$
\hat{f}(y,\delta,t) := f\!\left(\Pi^{-1}(y,\delta),t\right) \qquad \forall (y,\delta,t) \in \Sigma \times (0,\epsilon) \times (0,T) \, ,
$$
namely the original function $f$ written in terms of the above coordinate system, belongs to the space $ L^1(\Sigma \times (0,\epsilon) \times (0,T)) $ with respect to the product measure $ d\Sigma \otimes d\delta \otimes dt $, where $ d\Sigma $ stands for the $(n-1)$-dimensional Hausdorff measure on $ \Sigma $ and $ d\delta \otimes dt $ is the standard Lebesgue measure in $ (0,\epsilon) \times (0,T) $. Clearly the actual volume measure of $ \Omega_R^\epsilon $ with respect to $ (y,\delta) $ is not just $ d\Sigma \otimes d\delta $, but it is represented by a regular density $ \mathcal{A}(y,\delta) $ which is bounded and bounded away from zero. As a consequence, $ \hat{f} $ can also be seen as an element of the space $ L^1( (0,\epsilon); L^1(\Sigma \times (0,T) ) ) $; hence, thanks to the Lebesgue differentiation theorem for vector-valued functions \cite[Theorem K.5]{AK}, we deduce that for almost every $ \delta_0 \in (0, \epsilon ) $
\begin{equation*}\label{lebesgue-diff}
\lim_{h \to 0^+} \frac{1}{h} \int_{\delta_0+h}^{\delta_0+2h} \int_0^T \int_{\Sigma} \left| \hat f(y,\delta,t) - \hat f(y,\delta_0,t) \right| d\Sigma dt d\delta =  0 \, ,
\end{equation*}
which is in fact equivalent to
\begin{equation}\label{lebesgue-diff-1}
\lim_{h \to 0^+} \frac{1}{h} \int_{\delta_0+h}^{\delta_0+2h} \int_0^T \int_{\Sigma} \left| \hat f(y,\delta,t) - \hat f(y,\delta_0,t) \right| \mathcal{A}(y,\delta) \, d\Sigma dt d\delta =  0 \, .
\end{equation}
If we let $ \Sigma_0 $ denote the set of points of $ \Omega_R^{\epsilon} $ at distance $ \delta_0 $ from $ \Sigma $, by $ \pi_0(x) $ the projection of $ x \in \Omega_R^\epsilon $ onto $ \Sigma_0 $ and by $ \Sigma_0^h $ the set
$$
\Sigma_0^h := \left\{ x \in \Omega_R^\epsilon : \  h < \operatorname{d}(x,\pi_0(x))  <  2h \right\} \cap  \left\{ x \in \Omega_R^\epsilon : \   \operatorname{d}(x,\pi(x)) > \delta_0 \right\}  ,
$$
then \eqref{lebesgue-diff-1} can be rewritten as
\begin{equation*}\label{lebesgue-diff-2}
\lim_{h \to 0^+} \frac{1}{\mu\!\left(\Sigma_0^h\right)} \int_0^T \int_{\Sigma_0^h} \left| f(x,t) - f(\pi_0(x),t) \right| d\mu dt = 0 \, .
\end{equation*}
Indeed, letting $ \delta $ vary at a fixed $ y \in \Sigma $ is equivalent to moving along the geodesic given by the inward normal direction at $ (y,0) \equiv \pi(x) $ (see e.g.~\cite[Theorem 4.8 (12)]{Fed} in the Euclidean setting), so that the point identified by $ (y,\delta) $ actually represents the projection of $ x $ onto the submanifold $ \Pi^{-1}\!\left( \Sigma \times \{ \delta \} \right) $, for all $ \delta \in (0,\epsilon) $.

Given any (small enough) $ h>0 $, let $ \psi_h:[0,\infty) \to [0,\infty) $ be a smooth cut-off function enjoying the following properties:
\begin{equation}\label{cutoff1}
 0 \le \psi_h \le 1  \, , \qquad \psi_h=0 \quad \text{in } [0,\delta_0 + h] \, , \qquad \psi_h = 1 \quad \text{in } [\delta_0+2h,\infty)
\end{equation}
and
\begin{equation}\label{cutoff1-bis}
\left\| \psi_h^\prime \right\|_\infty \le \frac{c}{h} \, , \qquad \left\| \psi_h^{\prime\prime} \right\|_\infty \le \frac{c}{h^2} \, ,
\end{equation}
for a suitable constant $ c>0 $ independent of $ h $. We can then transplant $ \psi_h $ onto $ \Omega_R^\epsilon $ by setting
\begin{equation*}\label{cutoff2}
\phi_h(x) :=  \psi_h(\mathrm{d}(x,\pi(x)))  \qquad \forall x \in \Omega_R^\epsilon \, ,
\end{equation*}
that is, in the chosen coordinate frame,
\begin{equation*}\label{cutoff3}
\phi_h\!\left(\Pi^{-1}(y,\delta)\right) =:\hat{\phi}_h(y,\delta) = \psi_h(\delta) \qquad \forall (y,\delta) \in \Sigma \times (0,\epsilon) \, .
\end{equation*}
Because $ \delta $ is a geodesic coordinate in $ \Omega_R^\epsilon $ and $ \hat{\phi}_h $ only depends on $ \delta $, the Laplace-Beltrami operator applied to $ \phi_h $ reads (we refer e.g.~to \cite[Formula (3.35)]{BMMP} or to the proof of \cite[Lemma 2.13]{MRS})
\begin{equation}\label{lap-1}
\Delta \phi_h(x) = \psi_h^{\prime\prime}(\delta) +  \mathsf{m}(y,\delta) \, \psi_h^\prime(\delta) \qquad \forall x \equiv (y,\delta) \in \Sigma \times (0,\epsilon) \, ,
\end{equation}
where $ \mathsf{m}(\delta,y) $ is precisely the Laplace-Beltrami operator applied to the distance function $ \delta \equiv \mathrm{d}(x,\pi(x)) $, which is in fact a regular function that coincides with the partial derivative of $ (y,\delta) \mapsto \log \mathcal{A}(y,\delta) $ with respect to $ \delta $. If $ g $ is a $ C^2\big(\big(\Omega_R^\epsilon \setminus \Omega_R^{\delta_0}\big) \times [0,T] \big) $ function and $ \hat{g} $ is its representative with respect to the coordinates $(y,\delta)$ in $ \Omega_R^\epsilon $, it follows that
\begin{equation}\label{grad-grad}
\left\langle \nabla  \phi_h(x)  \, ,  \nabla g(x,t) \right\rangle = \psi_h^\prime(\delta) \, \frac{\partial \hat g}{\partial \delta}(y,\delta,t)  \qquad \forall (x,t) \equiv (y,\delta,t) \in \Sigma \times [\delta_0,\epsilon) \times [0,T] \, ,
\end{equation}
where $ \langle , \rangle $ stands for the inner product in the tangent space of $M$ at $ x $. On the other hand, by the product rule we have
\begin{equation}\label{product}
\Delta \! \left( \phi_h \, g \right) = \phi_h \, \Delta g + 2 \left\langle \nabla \phi_h  \, , \nabla g \right\rangle + g \, \Delta \phi_h \qquad \text{in } \big(  \Omega_R^\epsilon \setminus \Omega_R^{\delta_0} \big)  \times [0,T] \, .
\end{equation}
Let $ D:=\Omega_R \setminus \overline{\Omega}_{R}^{\delta_0} $, so that $ \partial D = \Sigma_0 $, and replace for the moment $ D_k $ with $ D $ in the statement (at the end of the proof we will explain how to pick the sequence $D_k$). From here on we take for granted that the test function $ \eta $ as in the statement, along with $ \eta_t $, can be continuously extended to $0$ in $ \big(M \setminus \overline{D}\big) \times [0,T] $ and $ \phi_h $ can be smoothly extended to $ 0 $ in $ M \setminus \Omega_R $ and to $ 1 $ in $ \Omega_R \setminus \Omega_R^\epsilon $, provided $h$ is small enough. By applying \eqref{p1} to the (admissible) test function $ \xi = \phi_h \eta $, in view of \eqref{product} (with $ g = \eta $) we obtain
\begin{equation}\label{p1-step1}
\int_0^T \int_{M} f \left( \phi_h \, \Delta \eta + 2 \left\langle \nabla \phi_h  \, , \nabla \eta \right\rangle + \eta \, \Delta \phi_h \right) d\mu dt \le F\!\left( \phi_h \eta , \phi_h \eta_t \right) .
\end{equation}
By construction, it is plain that
\begin{equation}\label{p1-step2}
\lim_{h \to 0^+} \int_0^T \int_{M} f \, \phi_h \, \Delta \eta \, d\mu dt = \int_0^T \int_{D} f \, \Delta \eta \, d\mu dt \qquad \text{and} \qquad \lim_{h \to 0^+} F\!\left( \phi_h \eta , \phi_h \eta_t \right) = F\!\left(\overline{\eta},\overline{\eta}_t\right) .
\end{equation}
We now focus on the last two terms of the integral in the left-hand side of \eqref{p1-step1}. As for the first one, by virtue of \eqref{cutoff1} and \eqref{grad-grad} (still with $ g = \eta $) we deduce the identity
\begin{equation}\label{p1-step3}
\int_0^T \int_{M} f \left\langle \nabla \phi_h  \, , \nabla \eta \right\rangle  d\mu dt = \int_{\delta_0+h}^{\delta_0+2h} \int_0^T \int_{\Sigma} \hat f(y,\delta,t) \, \frac{\partial\hat\eta}{\partial \delta}(y,\delta,t) \, \psi_h^\prime(\delta)  \, \mathcal{A}(y,\delta) \, d\Sigma dt d\delta \, ,
\end{equation}
whose right-hand side can be rewritten in the following way:
\begin{equation}\label{p1-step4}
\begin{aligned}
& \int_{\delta_0+h}^{\delta_0+2h} \int_0^T \int_{\Sigma} \hat f(y,\delta,t) \, \frac{\partial \hat \eta}{\partial \delta}(y,\delta,t) \, \psi_h^\prime(\delta)  \, \mathcal{A}(y,\delta) \, d\Sigma dt d\delta \\
= &  \int_{\delta_0+h}^{\delta_0+2h} \int_0^T \int_{\Sigma} \left[ \hat f(y,\delta,t) - \hat f(y,\delta_0,t) \right] \frac{\partial\hat\eta}{\partial \delta}(y,\delta,t) \, \psi_h^\prime(\delta)  \, \mathcal{A}(y,\delta) \, d\Sigma dt d\delta \\
 & + \int_0^T \int_{\Sigma} \hat f(y,\delta_0,t) \left( \int_{\delta_0+h}^{\delta_0+2h} \, \frac{\partial\hat\eta}{\partial \delta}(y,\delta,t) \, \psi_h^\prime(\delta)  \, \mathcal{A}(y,\delta) \, d\delta \right) d\Sigma dt \, .
\end{aligned}
\end{equation}
Thanks to \eqref{cutoff1-bis} and \eqref{lebesgue-diff-1}, it is not difficult to check that the first integral in the right-hand side of \eqref{p1-step4} vanishes as $ h \to 0^+ $:
$$
\begin{aligned}
& \limsup_{h \to 0^+} \int_{\delta_0+h}^{\delta_0+2h} \int_0^T \int_{\Sigma} \left| \hat f(y,\delta,t) - \hat f(y,\delta_0,t) \right| \left| \frac{\partial\hat\eta}{\partial \delta}(y,\delta,t) \, \psi_h^\prime(\delta) \right| \mathcal{A}(y,\delta) \, d\Sigma dt d\delta \\
\le & \, c \left\| \nabla \eta \right\|_\infty \, \lim_{h \to 0^+} \frac{1}{h} \int_{\delta_0+h}^{\delta_0+2h} \int_0^T \int_{\Sigma} \left| \hat f(y,\delta,t) - \hat f(y,\delta_0,t) \right| \mathcal{A}(y,\delta) \, d\Sigma dt d\delta = 0 \, . \\
\end{aligned}
$$
On the other hand, because $ \mathcal{A} $ and $ \frac{\partial \hat \eta}{\partial \delta} $ are locally uniformly continuous in $ \Sigma \times [\delta_0,\epsilon) \times [0,T]  $ and $ \psi_h^\prime(\delta) $ is converging to a Dirac delta centered at $ \delta_0 $ as $ h \to 0^+ $, it follows that
$$
\lim_{h \to 0^+} \int_{\delta_0+h}^{\delta_0+2h} \, \frac{\partial\hat\eta}{\partial \delta}(y,\delta,t) \, \psi_h^\prime(\delta)  \, \mathcal{A}(y,\delta) \, d\delta  = \frac{\partial\hat\eta}{\partial \delta}(y,\delta_0,t) \, \mathcal{A}(y,\delta_0)
$$
uniformly on $ \Sigma \times [0,T] $. Since $ \hat f(\cdot,\delta_0,\cdot) \in L^1(\Sigma \times (0,T)) $, upon recalling \eqref{p1-step3} we can thus assert that
\begin{equation}\label{p1-step5}
\lim_{h \to 0^+} \int_0^T \int_{M} f \left\langle \nabla \phi_h  \, , \nabla \eta \right\rangle  d\mu dt = \int_0^T \int_{\Sigma} \hat f(y,\delta_0,t) \, \frac{\partial\hat\eta}{\partial \delta}(y,\delta_0,t) \, \mathcal{A}(y,\delta_0) \, d\Sigma dt \, .
\end{equation}
Let us turn to the last integral in the left-hand side of \eqref{p1-step1}, which is the most delicate term. By virtue of \eqref{cutoff1} and \eqref{lap-1}, we have:
\begin{equation}\label{p2-step1}
\begin{aligned}
& \int_0^T \int_{M} f \, \eta \, \Delta \phi_h \, d\mu dt \\
= & \int_{\delta_0+h}^{\delta_0+2h} \int_0^T \int_{\Sigma} \hat f(y,\delta,t) \, \hat{\eta}(y,\delta,t) \left[ \psi_h^{\prime\prime}(\delta) +  \mathsf{m}(y,\delta) \, \psi_h^\prime(\delta) \right] \mathcal{A}(y,\delta) \, d\Sigma dt d\delta \, .
\end{aligned}
\end{equation}
Since $ \hat{\eta} = 0 $ on $ \Sigma \times \{ \delta_0 \} \times [0,T] $ and $ \hat{\eta} \in C^1(\Sigma \times [\delta_0,\epsilon] \times [0,T]) $, the  following estimate holds:
\begin{equation}\label{p2-step2}
\left| \hat \eta (y,\delta,t) \right| \le 2h \left\| \nabla \eta \right\|_\infty \qquad \forall (y,\delta,t) \in \Sigma \times [\delta_0+h,\delta_0+2h] \times [0,T] \, ;
\end{equation}
furthermore, thanks to \eqref{cutoff1-bis} and the fact that $ \mathsf{m}(y,\delta) $ is a smooth function, there exists a positive constant independent of $ h $, which is still denoted by $ c $, such that
\begin{equation}\label{p2-step3}
\left| \psi_h^{\prime\prime}(\delta) \right| +  \frac{\left|  \mathsf{m}(y,\delta) \, \psi_h^\prime(\delta) \right|}{h} \le \frac{c}{h^2} \qquad \forall (y,\delta) \in \Sigma \times [\delta_0+h,\delta_0+2h] \, .
\end{equation}
Therefore, by combining \eqref{p2-step2}, \eqref{p2-step3} and reasoning as in \eqref{p1-step4}, it is apparent that
\begin{equation}\label{p2-step4}
\begin{aligned}
& \lim_{h \to 0^+} \int_{\delta_0+h}^{\delta_0+2h} \int_0^T \int_{\Sigma} \hat f(y,\delta,t) \, \hat{\eta}(y,\delta,t) \left[ \psi_h^{\prime\prime}(\delta) +  \mathsf{m}(y,\delta) \, \psi_h^\prime(\delta) \right] \mathcal{A}(y,\delta) \, d\Sigma dt d\delta \\
= & \lim_{h \to 0^+} \int_0^T \int_{\Sigma} \hat f(y,\delta_0,t) \left( \int_{\delta_0+h}^{\delta_0+2h} \, \hat{\eta}(y,\delta,t) \, \psi_h^{\prime\prime}(\delta) \, \mathcal{A}(y,\delta) \, d\delta \right) d\Sigma dt \, .
\end{aligned}
\end{equation}
An elementary integration by parts with respect to $ \delta $ yields (let us observe that by construction $ \psi_h^\prime(\delta_0+h)=\psi_h^\prime(\delta_0+2h)=0 $)
\begin{equation*}\label{p2-step5}
\begin{aligned}
& \int_{\delta_0+h}^{\delta_0+2h} \, \hat{\eta}(y,\delta,t) \, \psi_h^{\prime\prime}(\delta) \, \mathcal{A}(y,\delta) \, d\delta \\
= & - \int_{\delta_0+h}^{\delta_0+2h} \, \frac{\partial\hat{\eta}}{\partial \delta}(y,\delta,t) \, \psi_h^{\prime}(\delta) \, \mathcal{A}(y,\delta) \, d\delta - \int_{\delta_0+h}^{\delta_0+2h} \, \hat{\eta}(y,\delta,t) \, \psi_h^{\prime}(\delta) \, \frac{\partial\mathcal{A}}{\partial \delta}(y,\delta) \, d\delta \, ,
\end{aligned}
\end{equation*}
whence
\begin{equation}\label{p2-step6}
\lim_{h \to 0^+} \int_{\delta_0+h}^{\delta_0+2h} \, \hat{\eta}(y,\delta,t) \, \psi_h^{\prime\prime}(\delta) \, \mathcal{A}(y,\delta) \, d\delta = - \frac{\partial\hat\eta}{\partial \delta}(y,\delta_0,t) \, \mathcal{A}(y,\delta_0) \, ,
\end{equation}
still uniformly on $ \Sigma \times [0,T] $. As a result, thanks to \eqref{p2-step1} and \eqref{p2-step4}--\eqref{p2-step6} we can infer that
\begin{equation}\label{p2-step7}
\lim_{h \to 0^+} \int_0^T \int_{M} f \, \eta \, \Delta \phi_h \, d\mu dt = - \int_0^T \int_{\Sigma} \hat f(y,\delta_0,t) \, \frac{\partial\hat\eta}{\partial \delta}(y,\delta_0,t) \, \mathcal{A}(y,\delta_0) \, d\Sigma dt \, .
\end{equation}
Since $ \delta $ is the geodesic coordinate along the inward normal direction $ -\nu $ on $ \partial D $ (see again \cite[Theorem 4.8 (12)]{Fed}), the identity
\begin{equation}\label{p2-step8}
-\int_0^T \int_{\Sigma} \hat f(y,\delta_0,t) \, \frac{\partial\hat\eta}{\partial \delta}(y,\delta_0,t) \, \mathcal{A}(y,\delta_0) \, d\Sigma dt = \int_0^T \int_{\partial D} f \, \frac{\partial \eta}{\partial \nu} \, d\sigma dt
\end{equation}
holds, where $ d\sigma $ is precisely the $(n-1)$-dimensional Hausdorff measure on $ \partial D $. Hence by combining \eqref{p1-step1}, \eqref{p1-step2}, \eqref{p1-step5}, \eqref{p2-step7} and \eqref{p2-step8}, we end up with
\begin{equation}\label{p-last}
\int_0^T \int_{D} f \, \Delta \eta \, d\mu dt - \int_0^T \int_{\partial D} f \, \frac{\partial \eta}{\partial \nu} \, d\sigma dt \le F\!\left(\overline{\eta},\overline{\eta}_t\right) .
\end{equation}
We are left with providing an appropriate choice of the sets $ \{ D_k \} $ as in the statement. Let $ R_k \to \infty $ be a given strictly increasing sequence of regular values for the exhaustion function $ \mathcal{E} $. For each $k \in \mathbb{N} $, we define $ D_k $ to be the set $ D $ obtained in the above construction with $ R \equiv R_{k+1} $ and the further requirement that $ \epsilon \in (0,d_k) $, where $d_k$ stands for the (positive) minimum of the distance between points of the level sets $ \partial \Omega_{R_k} $ and $ \partial \Omega_{R_{k+1}} $. This ensures the validity of \eqref{double-inclusion}, whence $\{ D_k \} $ is indeed a regular exhaustion of $M$. Note that \eqref{p2} is just inequality \eqref{p-last} with $ D\equiv D_k $.
\end{proof}

For simplicity, in \eqref{p1} we required the test functions $ \xi $ to be compactly supported and $ C^2 $ instead of $ C^\infty $, so that the differential inequality is not exactly in distributional form. Nevertheless, a routine local approximation of $ C^2 $ functions with $ C^\infty $ functions allows one to establish the same result even starting from \eqref{p1} in distributional form.

\subsection{Construction of the candidate minimal solution}\label{constr-min}
Given a nonnegative initial datum $ u_0 \in L^1_{{\mathrm{loc}}}(M) $, an exhaustion function $ \mathcal{E} $ of $ M $ and a corresponding strictly increasing sequence of regular values $ R_k \to \infty$, we set up the following ``lifted'' and bounded Dirichlet problems:
\begin{equation}\label{epsilon-R}
\begin{cases}
u_t = \Delta u^m & \text{in } \Omega_{R_k} \times (0,+\infty) \, , \\
u = \ell & \text{on } \partial \Omega_{R_k} \times (0,+\infty) \, , \\
u =  \ell + u_0 \wedge \beta   & \text{on } \Omega_{R_k} \times \{ 0 \} \, ,
\end{cases}
\end{equation}
where $ \ell,\beta >0 $ and we recall that the domains $ \{ \Omega_{R_k} \} $ are defined as in \eqref{sub}.  Note that in this way the equation implicitly becomes both \emph{nondegenerate} and \emph{nonsingular} (i.e.~\emph{quasilinear}), hence it enjoys several regularity properties. In particular, standard comparison principles applied to the corresponding solutions, which are denoted by $ \{  u_{k,\ell,\beta} \}  $, ensure that
\begin{equation}\label{ordering}
\begin{gathered}
\ell \le u_{k,\ell,\beta} \le u_{k+1,\ell,\beta} \le \ell + \beta \, , \qquad u_{k,\ell,\beta} \le u_{k,\ell',\beta} \, , \qquad u_{k,\ell,\beta} \le u_{k,\ell,\beta'} \\
 \forall \ell'>\ell>0 \, , \quad \forall \beta'>\beta>0 \, , \quad \forall k \in \mathbb{N} \, .
\end{gathered}
\end{equation}
We now show that the nonnegative function $\underline{u}$ defined as
\begin{equation}\label{ordering2}
 \underline{u} := \lim_{k \uparrow \infty} u_k \, , \qquad  u_k := \lim_{\ell \downarrow 0} u_{k,\ell}  \, , \qquad u_{k,\ell} := \lim_{\beta \uparrow \infty} u_{k,\ell,\beta} \, ,
\end{equation}
is a solution of \eqref{cauchy}, where each $ u_k $ is tacitly extended to zero outside $ \Omega_{R_k} $. To this aim, first of all we observe that the \emph{very weak} version of \eqref{epsilon-R} entails
\begin{equation}\label{weakfor-1}
\begin{gathered}
\int_{0}^{+\infty} \int_{\Omega_{R_k}} u_{k,\ell,\beta} \, \xi_t \, d\mu dt + \int_{0}^{+\infty} \int_{\Omega_{R_k}} u_{k,\ell,\beta}^m \, \Delta \xi \, d\mu dt + \int_{\Omega_{R_k}}  \left[ \ell + u_0(x) \wedge \beta \right] \xi(x,0) \, d\mu = 0 \\
\forall \xi \in C^\infty_c\!\left(\Omega_{R_k} \times [0,+\infty) \right) .
\end{gathered}
\end{equation}
By virtue of Proposition \ref{prop1} applied to $ u \equiv u_{k,\ell,\beta}  $, $ v = 0 $ and $ s=0 $, we infer that
\begin{equation}\label{eq:hp-1}
\left[\int_{\Omega_R} u_{k,\ell,\beta}(x,t) \, d\mu \right]^{1-m} \leq \left[\int_{\Omega_{2R}} \left( \ell + u_0 \right)  d\mu \right]^{1-m} + \mathcal{H}_{R} \, t  \qquad \forall  t> 0 \, ,
\end{equation}
provided $ 0 < 2R < R_k $. Although Proposition \ref{prop1}  was stated for global solutions only, i.e.~existing in the whole of $M$, from the corresponding proof one can check that it is enough to ask that they are defined at least in $ \Omega_{2 R} $, whence the request $ 0 < 2R < R_k $. In fact we have also exploited \eqref{eq3} for \emph{every} $ t>0 $ and down to $ t=0 $: this is a simple consequence of the continuity of  $ t \mapsto u_{k,\ell,\beta}(\cdot,t) $ as a curve with values in $  L^1(\Omega_{R_k}) $, due to standard quasilinear theory (see e.g.~\cite[Chapter 3]{V}). We are therefore in position to pass to the limit as $ \beta \uparrow \infty $; inequalities \eqref{ordering} guarantee that $ \beta \mapsto u_{k,\ell,\beta} $ is increasing, hence it admits a pointwise limit $ u_{k,\ell} $ which, by \eqref{eq:hp-1} and monotone convergence, complies with
\begin{equation}\label{eq:hp-2}
\left[\int_{\Omega_R} u_{k,\ell}(x,t) \, d\mu \right]^{1-m} \leq \left[\int_{\Omega_{2R}} \left( \ell + u_0 \right)  d\mu \right]^{1-m} + \mathcal{H}_{R} \, t  \qquad \forall  t> 0 \, .
\end{equation}
We can then let $ \ell \downarrow 0 $ upon noticing that, still by \eqref{ordering}, the map $ \ell \mapsto u_{k,\ell} $ is decreasing. As a result it admits a pointwise limit $ u_{k} $ which, by \eqref{eq:hp-2}, satisfies
\begin{equation}\label{eq:hp-3}
\left[\int_{\Omega_R} u_{k}(x,t) \, d\mu \right]^{1-m} \leq \left[\int_{\Omega_{2R}}  u_0 \,  d\mu \right]^{1-m} + \mathcal{H}_{R} \, t  \qquad \forall  t> 0 \, .
\end{equation}
In fact each $ u_k $ is a solution of the \emph{homogeneous} Dirichlet problem
\begin{equation}\label{hom-R}
\begin{cases}
u_t = \Delta u^m & \text{in } \Omega_{R_k} \times (0,+\infty) \, , \\
u = 0 & \text{on } \partial \Omega_{R_k} \times (0,+\infty) \, , \\
u =  u_0  & \text{on } \Omega_{R_k} \times \{ 0 \} \, .
\end{cases}
\end{equation}
Finally, we deal with the passage to the limit as $ k \to \infty $, which is the most delicate. We can claim again, by virtue of \eqref{ordering}, that the sequence $ \{ u_k \} $ is monotone increasing, so that  by letting $ k \to \infty $ in \eqref{eq:hp-3}  we end up with
\begin{equation}\label{eq:hp-4}
\left[\int_{\Omega_R} \underline{u}(x,t) \, d\mu \right]^{1-m} \leq \left[\int_{\Omega_{2R}}  u_0 \,  d\mu \right]^{1-m} + \mathcal{H}_{R} \, t  \qquad \forall  t> 0 \, ,
\end{equation}
where $ \underline{u} $ is the pointwise limit given in \eqref{ordering2}. From \eqref{weakfor-1} and the previous passages to the limit, it is plain that
\begin{equation*}\label{weakfor-2}
\begin{gathered}
\int_{0}^{+\infty} \int_{M} u_{k} \, \xi_t \, d\mu dt + \int_{0}^{+\infty} \int_{M} u_{k}^m \, \Delta \xi \, d\mu dt + \int_{M} u_0(x) \, \xi(x,0) \, d\mu = 0 \\
\forall \xi \in C^\infty_c\!\left(\Omega_{R_k} \times [0,+\infty) \right) .
\end{gathered}
\end{equation*}
On the other hand, estimate \eqref{eq:hp-4} yields $ \underline{u}  \in L^1_{\mathrm{loc}}(M\times [0, +\infty)) $, hence $ \{ u_k \} $ converges to $ \underline{u}  $ at least in $L^1_{\mathrm{loc}}(M\times [0, +\infty))$ (recall that $R>0$ is any number smaller than $ R_k/2 $ and $ R_k \to \infty $), therefore we are allowed to let $ k \to \infty $ to obtain
\begin{equation}\label{weakfor-3}
\begin{gathered}
\int_{0}^{+\infty} \int_{M} \underline{u}  \, \xi_t \, d\mu dt + \int_{0}^{+\infty} \int_{M} \underline{u} ^m \, \Delta \xi \, d\mu dt + \int_{M} u_0(x) \, \xi(x,0) \, d\mu = 0 \\
\forall \xi \in C^\infty_c\!\left(M\times [0,+\infty) \right) .
\end{gathered}
\end{equation}
A routine time cut-off argument then ensures that \eqref{weakfor-3} is equivalent to \eqref{distrib-id}--\eqref{eq12c}, so that $ \underline{u}  $ is indeed a solution of \eqref{cauchy} in the sense of Definition \ref{defsol}.

\subsection{Local comparison and minimality}\label{local-comp}

By carefully adapting a strategy introduced in \cite{PZ} for the fast diffusion equation with absorption in the Euclidean space $\mathbb R^n$, first of all we show a local comparison result between the approximate solutions $ \{ u_k \} $ of \eqref{hom-R} and a general nonnegative solution of \eqref{cauchy}. This is the fundamental tool we need in order to prove Theorem \ref{distr-sol}.
\begin{lemma}\label{loc-comp}
	Let $u_0\in L^2_{{\mathrm{loc}}}(M)$, with $u_0\ge0$. Let $u\in L^2_{\mathrm{loc}}(M\times[0,+\infty))$ be a nonnegative distributional solution of problem \eqref{cauchy}, in the sense of Definition \ref{defsol}. Given $ k \in \mathbb{N} $, let $ u_k $ be the solution of \eqref{hom-R} constructed in Subsection \ref{constr-min}. Then $ u \ge u_k $ almost everywhere in $ \Omega_{R_k} \times (0,+\infty) $.
\end{lemma}
\begin{proof}
To begin with, consider the solution $u_{k, \ell,\beta}$ of the approximate problem \eqref{epsilon-R} and, with some abuse of notation, let the same symbol denote its extension to the whole $ M \times (0,+\infty) $ obtained by setting $u_{k, \ell,\beta} = \ell $ in $ \big( M \setminus \Omega_{R_k} \big) \times (0,+\infty) $. Because $ u_{k, \ell,\beta} \ge \ell $ in $  \Omega_{R_k} \times (0,+\infty) $ and $ u_{k, \ell,\beta} = \ell $ on $ \partial \Omega_{R_k} \times (0,+\infty) $, it is not difficult to check that such an extension in fact becomes a subsolution to the same problem in the whole manifold. As a consequence, we infer the validity of the following inequality:
\begin{equation}\label{n3}\begin{aligned}
&\int_0^T \int_M \left[ \left(u - u_{k, \ell,\beta}\right) \xi_t + \left({u}^m - u_{k, \ell,\beta}^m \right) \Delta \xi \right] d\mu dt\\
\le & \int_M \left[ u(x,T) - u_{k, \ell,\beta}(x,T) \right] \xi(x,T) \, d\mu \,+ \ell \int_M \xi(x,0) \, d\mu \, ,
\end{aligned}\end{equation}
for almost every $ T>0 $ and every nonnegative $\xi\in C^\infty_c(M\times [0, T])$. We point out that \eqref{n3} can easily be deduced by a time cut-off argument from the distributional versions of $ \partial_t u = \Delta u^m $ and $ \partial_t u_{k, \ell,\beta} \le \Delta u_{k, \ell,\beta}^m $; moreover, the negligible set of times $T$ for which \eqref{n3} may not hold depends only on $ u $, since as observed in Subsection \ref{constr-min} each solution $u_{k, \ell,\beta}$ is $ L^1 $ time continuous. We are therefore in position to apply Proposition \ref{lem:comp} to \eqref{n3}. As a result, there exists a regular exhaustion $ \{ D_k \} \subset M $ satisfying \eqref{double-inclusion} such that
\begin{equation}\label{n3bis}
\begin{aligned}
&\int_0^T \int_{D_k} \left[ \left(u - u_{k, \ell,\beta}\right) \xi_t + \left({u}^m - u_{k, \ell,\beta}^m \right) \Delta \xi \right] d\mu dt - \int_0^T \int_{\partial D_k} \left( u^m - \ell^m \right) \frac{\partial \xi}{\partial \nu} \, d\sigma dt \\
\le & \int_{D_k} \left[ u(x,T) - u_{k,\ell,\beta}(x,T) \right] \xi(x,T) \, d\mu \, + \ell \int_{D_k} \xi(x,0) \, d\mu
\end{aligned}
\end{equation}
for every nonnegative $\xi\in C^2\!\left(\overline{D}_k\times [0, T]\right)$ that vanishes on $ \partial D_k \times [0,T] $ (here for notational simplicity we do not relabel $ \xi $). It is plain that the normal derivative of any such $ \xi $ is nonpositive on $ \partial D_k \times [0,T] $, whence \eqref{n3bis} entails
\begin{equation}\label{n3ter}
\begin{aligned}
&\int_0^T \int_{D_k} \left[ \left(u - u_{k, \ell,\beta}\right) \xi_t + \left({u}^m - u_{k, \ell,\beta}^m \right) \Delta \xi \right] d\mu dt + \ell^m \int_0^T \int_{\partial D_k} \frac{\partial \xi}{\partial \nu} \, d\sigma dt \\
\le & \int_{D_k} \left[ u(x,T) - u_{k,\ell,\beta}(x,T) \right] \xi(x,T) \, d\mu \, + \ell \int_{D_k} \xi(x,0) \, d\mu \, .
\end{aligned}
\end{equation}
Let us introduce the following function:
\begin{equation*}\label{n4}
a(x,t):=
\begin{cases}
\frac{u^m(x,t) - u_{k, \ell,\beta}^m(x,t)}{u(x,t) - u_{k, \ell,\beta}(x,t)} & \text{if } u (x,t) \neq u_{k, \ell,\beta}(x,t) \, , \\
0 & \text{if }  u(x,t) = u_{k, \ell,\beta}(x,t) \, .
\end{cases}
\end{equation*}
Clearly $ a \ge 0 $. Besides, since $u\ge0$ and $u_{k, \ell,\beta}\ge \ell $, we have that
\begin{equation*}\label{a-bound}
a(x,t) = \frac{1}{u_{k, \ell,\beta}^{1-m}} \, \frac{\left(\frac u{u_{k, \ell,\beta}}\right)^m-1}{\left(\frac u{u_{k, \ell,\beta}}\right)-1} \, \chi_{u(x,t) \neq u_{k,\ell,\beta}(x,t)} \le  \frac1{\ell^{1-m}} \qquad \text{for a.e.~} (x,t) \in M \times (0,T) \, ,
\end{equation*}
upon noticing that
$$
\sup_{z \ge 0 , \,  z \neq 1} \frac{z^m-1}{z-1}=1 \, .
$$
 As a consequence, inequality \eqref{n3ter} can be rewritten in the following way:
\begin{equation}\label{n3quater}
\begin{aligned}
& \int_0^T \int_{D_k} \left(  u - u_{k, \ell,\beta} \right) \left(\xi_t + a \, \Delta \xi \right) d\mu dt + \ell^m \int_0^T \int_{\partial D_k} \frac{\partial \xi}{\partial \nu} \, d\sigma dt \\
\le &  \int_{D_k} \left[ u(x,T) - u_{k, \ell,\beta}(x,T) \right] \xi(x,T) \, d\mu + \ell\int_{D_k} \xi(x,0) \, d\mu \, .
\end{aligned}
\end{equation}
Thanks to the fact that $a$ is nonnegative and bounded, one can pick a sequence of strictly positive, bounded and smooth functions $ \{ a_h \} \subset C^\infty\!\left( M \times [0,T] \right) $ such that
\begin{equation}\label{Linf-star}
 \lim_{h \to \infty} \frac{\left| a - a_h \right|^2}{a_h} = 0 \qquad \text{in } \left( L^\infty\!\left( D_k \times (0,T) \right) \right)^* .
\end{equation}
For an explicit construction of an analogous sequence, see e.g.~the proof of \cite[Theorem 2.3]{GMP}. Given a nonnegative function $\omega\in C^\infty_c(D_k)$ and $ h \in \mathbb{N} $, consider now the solution $\xi_{h}$ of the backward parabolic (dual) problem
\begin{equation}\label{n9}
\begin{cases}
\xi_t + a_{h} \, \Delta \xi = 0 & \text{in } D_k \times (0, T) \, , \\
\xi = 0 & \text{on } \partial D_k \times (0, T) \, , \\
\xi = \omega & \text{on } D_k \times \{T\} \, .
\end{cases}
\end{equation}
In view of standard parabolic regularity, we can claim that $ \xi_{h} $ is smooth in $ \overline{D}_k \times [0,T] $.  Moreover, by the comparison principle it follows that
\begin{equation}\label{n43}
 0 \le \xi_{h} \le \left\| \omega \right\|_\infty \qquad \text{in } \overline{D}_k \times [0,T] \, .
\end{equation}
In order to bound the second term in the left-hand side of \eqref{n3quater} with $\xi \equiv \xi_{h}$, we need to estimate
$$
\left|\frac{\partial \xi_{h}}{\partial \nu}(x,t)\right| \quad \text{for every } x \in \partial D_k \text{ and } t \in (0,T) \, .
$$
To this aim, let $ G_k(x,y) $ denote the Green function of the Dirichlet Laplace-Beltrami operator in $ D_k $ and pick a point $ y_0 \in \Omega_{R_k} $.  It is apparent that $ x \mapsto G_k(x,y_0) $ is smooth and positive in $ {D}_{k} \setminus \Omega_{R_k} $; in particular, there exists $ \lambda>0 $ such that
\begin{equation}\label{eq:green-2}
 \lambda \, G_k(x,y_0) \ge \| \omega \|_\infty \quad \forall x \in \partial \Omega_{R_k} \qquad \text{and} \qquad \lambda \, G_k(x,y_0) \ge \omega(x) \quad \forall x \in D_k \setminus \Omega_{R_k} \, .
\end{equation}
Hence, by virtue of \eqref{n43} and \eqref{eq:green-2}, we infer that $ x \mapsto \lambda \, G_k(x,y_0) $ is a supersolution to the problem
\begin{equation*}\label{n9-bis}
\begin{cases}
v_t + a_{h} \, \Delta v = 0 & \text{in } \big( D_k  \setminus \overline{\Omega}_{R_k} \big) \times (0, T) \, , \\
v = 0 & \text{on } \partial D_k \times (0, T) \, , \\
v  = \xi_{h} & \text{on } \partial \Omega_{R_k} \times (0,T)  \, , \\
v = \omega & \text{on } \big( D_k  \setminus \overline{\Omega}_{R_k} \big) \times \{T\} \, ,
\end{cases}
\end{equation*}
whereas $ \xi_h $ is the solution of the same problem. Hence, we deduce that $ \xi_h \le \lambda \, G_k(\cdot,y_0) $ in $ \big( D_k  \setminus \overline{\Omega}_{R_k} \big) \times (0, T) $. Since both $ x \mapsto G_k(x,y_0) $ and $ \xi_h $ vanish on $ \partial D_k $, this implies
 \begin{equation}\label{der-norm-1}
\left| \frac{\partial \xi_{h}}{\partial \nu}(x,t) \right| \leq \lambda \left|\frac{\partial G_k}{\partial \nu} (x,y_0) \right| \le \lambda \max_{x \in \partial D_k} \left|\frac{\partial G_k}{\partial \nu} (x,y_0) \right| =: \widetilde{\lambda} \qquad \forall (x,t) \in \partial D_k \times (0,T) \, .
\end{equation}
Having estimated the normal derivative of $ \xi_h $, we are able to complete the proof. First of all, multiplying the differential equation in \eqref{n9} by $\Delta\xi_{h}$ and integrating by parts in $D_k\times (0, T)$, we obtain:
\begin{equation*}\label{energy}
\frac 1 2\int_{D_k} \left| \nabla\xi_{h}(x,0) \right|^2 d\mu + \int_0^T \int_{D_k} a_{h} \left| \Delta \xi_{h} \right|^2 d\mu dt = \frac 1 2 \int_{D_k} \left| \nabla \omega \right|^2 d\mu \, ,
\end{equation*}
so that
\begin{equation}\label{n48}
\begin{aligned}
&\left| \int_0^T\int_{D_k}\left(  u - u_{k,\ell,\beta} \right)\left(a-a_{h}\right) \Delta\xi_{h} \, d\mu d t\right| \\
\le  & \left[ \int_0^T\int_{D_k}\left|u-u_{k,\ell,\beta}\right|^2 \frac{\left|a-a_{h}\right|^2}{a_{h}}\,{d}\mu{d}t \right]^{\frac12}\left[\int_0^T\int_{D_k} a_{h}\left|\Delta\xi_{h}\right|^2 {d}\mu{ d}t \right]^{\frac12}  \\
\le & \, \frac{\left\| \nabla \omega  \right\|_2}{\sqrt 2} \left[ \int_0^T\int_{D_k}\left|u-u_{k,\ell,\beta}\right|^2 \frac{\left|a-a_{h}\right|^2}{a_{h}}\,{d}\mu{d}t \right]^{\frac12} .
\end{aligned}
\end{equation}
Since $ u_{k,\ell,\beta} \in L^\infty\!\left( D_k \times (0,T) \right) $ and $ u \in L^2\!\left( D_k \times (0,T) \right) $ by assumption, thanks to \eqref{Linf-star} we can infer that
\begin{equation}\label{bound}
\lim_{h \to \infty}  \int_0^T\int_{D_k}\left|u-u_{k,\ell,\beta}\right|^2 \frac{\left|a-a_{h}\right|^2}{a_{h}}\,{d}\mu{d}t = 0 \, .
\end{equation}
If we go back to \eqref{n3quater} with $\xi \equiv \xi_{h}$ and let $ h \to \infty $, recalling \eqref{n9}, \eqref{n43}, \eqref{der-norm-1}, \eqref{n48} and \eqref{bound}, we end up with
\begin{equation}\label{TTT}
 - \widetilde \lambda \, T \, \sigma(\partial D_k) \, \ell^m \le \int_{D_k} \left[ u(x,T) - u_{k,\ell,\beta}(x,T) \right] \omega(x) \, d\mu + \left\| \omega \right\|_\infty \mu(D_k) \, \ell \, .
\end{equation}
Upon letting first $ \beta \uparrow \infty $ and then $ \ell \downarrow 0 $ as in Subsection \ref{constr-min}, from \eqref{TTT} it follows that
\[
\int_{D_k} \left[ u(x,T) - u_{k}(x,T) \right] \omega(x) \, d\mu\ge0 \, .
\]
The thesis is therefore established in view of the arbitrariness of $ T $ and the test function $ \omega $, along with the inclusion $\Omega_{R_k}\subset D_k$.
\end{proof}

\begin{proof}[Proof of Theorem \ref{distr-sol}]
Let $ \underline{u} \ge 0 $ be defined by \eqref{ordering2}, namely the monotone limit of the solutions $\{ u_k \}$ of problems \eqref{hom-R} constructed in Subsection \ref{constr-min}. We have already established that $ \underline{u} $ is indeed a solution of \eqref{cauchy} in the sense of Definition \ref{defsol}. The fact that $\underline{u}$ belongs to the space $ L^2_{\mathrm{loc}}(M\times[0,+\infty)) $ can be deduced by means of an adaptation of the Herrero-Pierre estimates \eqref{eq:hp-4} to local $L^p$ norms ($ p>1 $), that follows similarly to the proof of \cite[Theorem 2.3]{BV2}. By passing to the limit as $ k \to \infty $ in the inequality $ u \ge u_k $ (let $ u_k $ be set to zero outside $ \Omega_{R_k} $), guaranteed by Lemma \ref{loc-comp}, we finally infer that $ u \ge \underline{u} $, whence the minimality property of $ \underline{u} $.
\end{proof}

\section{Proof of the nonexistence results for the elliptic equation} \label{Ellptic}
Prior to the proof of Theorem \ref{tell}, we establish two fundamental preliminary results. The first one is an adaptation, to the type of solutions we deal with, of the mean-value inequality for subharmonic functions due to Li and Schoen \cite[Theorem 2.1]{LS}.

Hereafter $ x \mapsto r(x) $ stands for the distance function from any \emph{fixed} reference point $ o \in M $.

\begin{proposition}\label{vwmv}
Let $w\in L^1_{\mathrm{loc}}(M)$ be a nonnegative function satisfying
\begin{equation}\label{eeesub}
\Delta w \geq 0 \qquad \text{in } \mathcal D'(M) \, .
\end{equation}
{Let $ \Omega $ be a regular precompact domain of $ M $. Let $ R>0 $ fulfill $ B_{5R}(o) \Subset \Omega $ and set $ K:= \inf_{x \in \Omega} \mathrm{Ric}(x) $. Then there exists a constant $ c(n,R,K) >0 $ such that}
\begin{equation}\label{eq16-tris}
\underset{x \in B_{\!\frac{R}{2}}\!(o)}{\operatorname{ess} \sup} \, w(x) \leq \frac{c(n,R,K)}{ \mu\!\left(B_{R}(o)\right) } \, \int_{B_R(o)} w \,{d}\mu  \, .
\end{equation}
\end{proposition}
\begin{proof}
First of all, we observe that \cite[Theorems 1.2 and 2.1]{LS} remain true for local {weak} subharmonic functions as well, i.e.~nonnegative functions belonging to the Sobolev space $ W^{1,2}_{\mathrm{loc}}(\Omega) $ and satisfying \eqref{eeesub} in $ \Omega $. This can easily be checked from the corresponding proofs. On the other hand, as we will explain below, it is possible to construct a sequence $ \{ w_h \} $ such that $ w_h \ge 0 $, $ w_h \in W^{1,\infty}_{\mathrm{loc}}(\Omega) $, $w_h$ is subharmonic in $ \Omega $ and $ w_h \to w $ as $ h \to \infty $ in $ L^1(\Omega) $ and almost everywhere in $ \Omega $. According to the above observations, by applying \cite[Theorem 2.1]{LS} with $ \tau = {1}/{4} $ to each $ w_h $ we obtain
\begin{equation}\label{eq16-bis}
\sup_{x \in B_{\!\frac{3R}{4}}\!(o)}  w_h(x) \leq  \frac{e^{c_n \log 4 \, \left( 1 + \sqrt{K} R \right) }}{\mu\!\left(B_{R}(o)\right)} \int_{B_R(o)} w_h \, d\mu \qquad \forall h \in \mathbb{N} \, ,
\end{equation}
where $ c_n $ is a positive constant depending only on $n$. Hence, by passing to the limit in \eqref{eq16-bis} as $ h \to \infty $, estimate \eqref{eq16-tris} follows.

In the sequel, we show how the sequence $ \{ w_h \} $ can be provided. By virtue of \eqref{eeesub}, we know that in fact $ \varsigma  := \Delta w  $ is a nonnegative Radon measure in $M$ (in particular finite in $ \Omega $). Moreover, by reasoning along the lines of the proof of Proposition \ref{lem:comp}, it is not difficult to see that $w$ solves the problem (up to possibly slightly enlarging $ \Omega $)
\begin{equation*}\label{meas-pb-4b}
\begin{cases}
\Delta w = \varsigma & \text{in } \Omega \, , \\
w = w |_{\partial \Omega} & \text{on } \partial \Omega  \, ,
\end{cases}
\end{equation*}
in the sense that
\begin{equation}\label{eq200}
\int_{\Omega} w \, \Delta \varphi \, {d}\mu -\int_{\partial\Omega} w|_{\partial \Omega}  \, \frac{\partial \varphi}{\partial \nu} \, {d}\sigma =\int_{\Omega} \varphi \, {d} \varsigma
\end{equation}
for every $\varphi\in C^{\infty}\!\left(\overline \Omega\right)$ with $\varphi=0$ on $\partial \Omega$. Note that $ w |_{\partial \Omega} \in L^1(\partial \Omega) $. We now let $ \mathsf{w} $ denote the solution of the following problem:
\begin{equation}\label{meas-pb-1}
\begin{cases}
\Delta \mathsf{w}  = 0 & \text{in }\Omega \, , \\
\mathsf{w}  = w |_{\partial \Omega} & \text{on } \partial \Omega \, .
\end{cases}
\end{equation}
In order to construct such a solution, we consider first the approximate problems
\begin{equation*}\label{meas-pb-1-aux}
\begin{cases}
\Delta \mathsf{w}_k = 0 & \text{in }\Omega \, , \\
\mathsf{w}_k = g_k & \text{on } \partial \Omega \, ,
\end{cases}
\end{equation*}
where $ g_k := w |_{\partial \Omega}  \wedge k  $. These solutions can in turn be obtained by approximating each $ g_k $ with regular boundary data, so that every $\mathsf{w}_k$ satisfies
\begin{equation}\label{eq201}
\int_{\Omega} \mathsf{w}_k \, \Delta\varphi\, {d}\mu -\int_{\partial\Omega} g_k \, \frac{\partial \varphi}{\partial \nu} \, {d}\sigma = 0
\end{equation}
for all test function $ \varphi $ as above. By the comparison principle we have $ 0 \le \mathsf{w}_k \le \mathsf{w}_{k+1} $ in $ \Omega $, whence $ \{ \mathsf{w}_k \} $ is monotone increasing. Let us pick $ \varphi $ as the solution of
\begin{equation*}\label{test-1}
\begin{cases}
-\Delta \varphi = 1 & \text{in }\Omega \, , \\
\varphi = 0 & \text{on } \partial \Omega \, .
\end{cases}
\end{equation*}
From \eqref{eq201} and the monotonicity of $ \{ \mathsf{w}_k \} $, we obtain:
\begin{equation*}\label{eq202}
\int_{\Omega} \left| \mathsf{w}_{k'} - \mathsf{w}_k \right| {d}\mu \le \left\| \tfrac{\partial \varphi}{\partial \nu} \right\|_{L^\infty(\partial \Omega)} \left\| g_{k'}-g_k \right\|_{L^1(\partial \Omega)} \qquad \forall k',k \in \mathbb{N} \, .
\end{equation*}
This shows that $ \{ \mathsf w_k \} $ is Cauchy in $ L^1(\Omega) $ and therefore converges to some function $ \mathsf{w} $ that satisfies
\begin{equation*}\label{eq203}
\int_{\Omega} \mathsf{w} \, \Delta\varphi \, {d}\mu -\int_{\partial\Omega} w |_{\partial \Omega} \, \frac{\partial \varphi}{\partial \nu} \, {d}\sigma = 0 \, ,
\end{equation*}
still for every $\varphi\in C^{\infty}\!\left(\overline \Omega\right)$ with $\varphi=0$ on $\partial \Omega$, namely \eqref{meas-pb-1}.

Let $ \{ \varsigma_h \} \subset C^\infty(M) $ be a sequence of nonnegative functions such that $ \varsigma_h \to \varsigma $ as $ h \to \infty $ vaguely in $ \Omega $ (i.e.~tested against any $C_0\!\left( \overline{\Omega} \right)$ function), with $ \varsigma_h(\Omega) = \varsigma(\Omega) $.  For each $ h \in \mathbb{N} $, we define $ v_h $ to be the solution of
\begin{equation*}\label{meas-pb-2}
\begin{cases}
\Delta v_h = \varsigma_h & \text{in } \Omega \, , \\
v_h = 0 & \text{on } \partial \Omega \, .
\end{cases}
\end{equation*}
By elliptic regularity results, for which we refer e.g.~to \cite[Theorem 1 and Subsection 4]{BG} in the Euclidean context (the fact that we work in a Riemannian framework is not relevant), we can claim that for every $ h \in \mathbb{N} $
$$
\left\| v_h \right\|_{W^{1,q}_0(\Omega)} \le C \, , \qquad \forall q \in \left(1,\tfrac{n}{n-1}\right) ,
$$
for some positive constant $ C $ depending only on $ q, \Omega, \varsigma(\Omega) $ (in particular independent of $h$).  As a consequence, up to a subsequence that we do not relabel, we can assert that $ \{ v_h \} $ converges in $ L^1(\Omega) $ and almost everywhere in $ \Omega $ to the solution $ \mathsf{v} $ of
\begin{equation*}\label{meas-pb-3}
\begin{cases}
\Delta \mathsf{v} = \varsigma & \text{in } \Omega \, , \\
\mathsf{v} = 0 & \text{on } \partial \Omega \, .
\end{cases}
\end{equation*}
We construct the above sequence $ \{ w_h \} $ by setting $ w_h := \left( \mathsf{w} + v_h \right)^+ $. Since both $ \mathsf{w} $ and $ v_h $ are regular inside $ \Omega $ and possess a nonnegative Laplacian, we deduce that each $ w_h  $ is locally Lipschitz and, thanks to Kato's inequality, weakly subharmonic in $\Omega$. We are left with showing that $ \{ w_h \} $ does converge to $ w $. To this end, first of all note that the sequence $ \{ \mathsf{w} + v_h \} $ converges in $ L^1(\Omega) $ and almost everywhere to the function $ v := \mathsf{w} + \mathsf{v} $, which solves
\begin{equation*}\label{meas-pb-4}
\begin{cases}
\Delta v = \varsigma & \text{in } \Omega \, , \\
v = w|_{\partial \Omega} & \text{on } \partial \Omega  \, ,
\end{cases}
\end{equation*}
in the sense of \eqref{eq200}, namely the same problem solved by $w$. This implies
\[
\int_{\Omega} \left( v - w \right) \Delta\varphi \, {d}\mu =0 \qquad \forall \varphi\in C^{\infty}\!\left(\overline \Omega\right)\!: \ \, \varphi |_{\partial \Omega} =0  \, .
\]
Given any $\psi\in C^\infty_c(\Omega)$, let us pick $\varphi$ as the solution of
\[
\begin{cases}
\Delta \varphi =  \psi & \text{in } \Omega \, , \\
\varphi =0 & \text{on } \partial \Omega \, .
\end{cases}
\]
It follows that
\[
\int_{\Omega} \left(v- w\right) \psi\, {d}\mu =0 \, ,
\]
whence  $ v=w $ in view of arbitrariness of $\psi$. We have therefore established that $ \{ \mathsf{w} + v_h \} $ converges in $ L^1(\Omega) $ and almost everywhere to $ w $; because $ w  $ is nonnegative, the same holds for the sequence $ \{ w_h \} $, and the proof is complete.
\end{proof}

Having in mind the original strategy of Osserman \cite{O} (see also \cite{Bre}), we now exhibit a suitable family of supersolutions to \eqref{eqn1} in balls that vanish as the corresponding radii go to infinity. To this purpose it is crucial to recall that, in view of assumption \eqref{eq1b}, the Laplacian-comparison theorem ensures that
\begin{equation}\label{neq7}
\Delta r(x) \leq (n-1) \, \frac{\psi'(r(x))}{\psi(r(x))}  \qquad \text{in } \mathcal D'(M) \, .
\end{equation}
{Note moreover that \eqref{neq7} also holds pointwise outside the cut locus of $o$. We refer for instance to \cite[Theorem 1.11]{MRS}}. 

\begin{lemma}\label{nlem}
Let the curvature condition \eqref{eq1b} be satisfied. Let
$$
H(r):=\int_0^r\frac{\int_0^\rho \psi(\zeta)^{n-1}\, d\zeta}{\psi(\rho)^{n-1}} \, d\rho \qquad \forall r \ge 0 \, .
$$
Given $ p>1 $ and $ R>0 $, there exists a constant $ \mathsf{C} >0$, depending only on $p$, such that the function
\begin{equation}\label{expl-form}
\overline{W}\!_R(x) := \mathsf{C} \, \frac{H(R)^\frac{1}{p-1}}{\left[ H(R)-H(r(x)) \right]^{\frac 2{p-1}}} \qquad \forall x \in B_R(o)
\end{equation}
fulfills
\begin{equation}\label{neq5}
\Delta \overline{W}\!_R \leq  \overline{W}\!_R^{\,\,p} \qquad \text{in } \mathcal D'\!\left(B_R(o)\right) .
\end{equation}
\end{lemma}
\begin{proof}
Thanks to the requirement $ \psi' \ge 0 $, it is readily seen that
\begin{equation}\label{neq6}
\left[ H'(r) \right]^2 \leq 2 H(r) \qquad \forall r \ge 0\,.
\end{equation}
{Indeed, one can rewrite $ H(r) $ as
\begin{equation}\label{neq88}
H(r) = \frac{\left[ H'(r) \right]^2}{2} + (n-1) \int_0^r \frac{\psi'(\rho)}{\psi(\rho)} \left[ H'(\rho) \right]^2  d\rho \qquad \forall r \ge 0\,.
\end{equation}
Note that, by virtue of the assumptions on $ \psi $, we have $ H \in C^2([0, \infty)) $.} 
With some abuse of notation, we tacitly use the identification $\overline{W}\!_R(x)\equiv \overline W\!_R(r(x)) \equiv \overline W\!_R(r)  $. The calculation of the derivatives of $ \overline W\!_R $ yields, for all $ 0 \le  r <R  $,
\[
\overline{W}\!_R^{\,\prime}(r) =\frac {2 \mathsf{C}}{p-1} \, \frac{H(R)^{\frac 1{p-1}}}{\left[H(R)-H(r)\right]^{\frac{p+1}{p-1}}} \, H'(r) \, ,
\]
\[
\overline{W}\!_R^{\,\prime\prime}(r) =\frac {2 \mathsf{C}}{p-1} \, \frac{H(R)^{\frac 1{p-1}}}{\left[ H(R)-H(r) \right]^{\frac{p+1}{p-1}}} \, H''(r) + \frac{2(p+1)\mathsf{C}}{(p-1)^2} \, \frac{H(R)^{\frac 1{p-1}}}{\left[H(R)-H(r)\right]^{\frac{2p}{p-1}}} \left[H'(r)\right]^2  .
\]
Upon taking derivatives in \eqref{neq88} and dividing by $ H'(r) $, we deduce that
\begin{equation}\label{neq88-bis}
H''(r) + (n-1) \, \frac{\psi'(r)}{\psi(r)} \, H'(r) = 1 \qquad \forall r>0\, .
\end{equation}
Hence, in view of \eqref{neq6}, \eqref{neq88-bis} and the fact that $ \overline{W}\!_R^{\,\prime} \ge 0$ along with \eqref{neq7} (see in particular \cite[Lemma 1.12]{MRS}), we have:
\[
\begin{aligned}
 \Delta \overline{W}\!_R (x) \leq & \, \overline{W}\!_R^{\,\prime\prime}(r(x)) +(n-1) \, \frac{\psi'(r(x))}{\psi(r(x))} \, \overline{W}\!_R^{\,\prime}(r(x)) \\
= & \, \frac {2 \mathsf{C}}{p-1} \, \frac{H(R)^{\frac 1{p-1}}}{\left[ H(R)-H(r) \right]^{\frac{p+1}{p-1}}} + \frac{2(p+1)\mathsf{C}}{(p-1)^2} \, \frac{H(R)^{\frac 1{p-1}}}{\left[H(R)-H(r)\right]^{\frac{2p}{p-1}}} \left[H'(r)\right]^2  \\
\leq & \, \frac {2 \mathsf{C}}{p-1} \, \frac{H(R)^{\frac 1{p-1}}}{\left[ H(R)-H(r) \right]^{\frac{p+1}{p-1}}}+ \frac{4(p+1)\mathsf{C}}{(p-1)^2} \, \frac{H(R)^{\frac 1{p-1}}}{\left[H(R)-H(r)\right]^{\frac{2p}{p-1}}} \, H(r) \quad \text{in } \mathcal{D}'\!\left( B_R(o) \right) .
\end{aligned}
\]
As a result, recalling that $H(r)<H(R)$ for all $ 0 \le  r <R  $, in order for \eqref{neq5} to hold it is enough to ask that
\[
\frac {2}{p-1} \, \frac{H(R)^{\frac 1{p-1}}}{\left[H(R)-H(r)\right]^{\frac{p+1}{p-1}}} + \frac{4(p+1)}{(p-1)^2} \, \frac{H(R)^{\frac p{p-1}}}{\left[H(R)-H(r)\right]^{\frac{2p}{p-1}}} \leq \mathsf{C}^{p-1} \, \frac{H(R)^\frac{p}{p-1}}{\left[H(R)-H(r)\right]^{\frac{2p}{p-1}}} \, ,
\]
which is equivalent to
\[
\frac {2}{p-1} \, H(R)^{\frac 1{p-1}} \left[H(R)-H(r)\right] + \frac{4(p+1)}{(p-1)^2} \, H(R)^{\frac p{p-1}} \leq \mathsf{C}^{p-1} H(R)^\frac{p}{p-1} \qquad \forall r \in  [0,R) \, .
\]
It is immediate to check that the latter inequality is satisfied provided
$$
\mathsf{C} \ge \left[ 2\,\frac{3p+1}{(p-1)^2} \right]^{\frac{1}{p-1}} \, .
$$
\end{proof}

The main idea lying behind the proof of Theorem \ref{tell} consists of comparing each supersolution $\overline{W}\!_R$ constructed in Lemma \ref{nlem} to the global nonnegative subsolution $W$ in $B_R(o)$. Since $\lim_{r(x)\to R^-}\overline{W}\!_R(r(x))=\infty$ whereas $W$ turns out to be locally bounded, by the comparison principle it follows that $W\leq \overline{W}\!_R$ in $B_R(o)$. This clearly implies that $W$ is identically zero upon letting $R\to \infty$ and using \eqref{eq1c}. Unfortunately we cannot directly exploit this procedure, because $\partial B_R(o)$ may not be smooth enough (in general it is only Lipschitz) and $W$ may not have a trace on $\partial B_R(o)$. In order to overcome such difficulties we need to suitably approximate $ B_R(o) $ by the sublevel sets $\Omega_{R, \varepsilon}$ introduced in \eqref{neq20}.

\begin{proof}[Proof of Theorem \ref{tell}]
We will consider the case $  \alpha = 1 $ only, which is not restrictive by virtue of the standard change of variables $ {W}_\alpha = \alpha^{-{1}/{(p-1)}} \, W $.

\noindent\textit{(i)} Let $W\geq 0$ be a locally integrable function satisfying \eqref{eqn1-subsol}. In particular $ W $ is subharmonic, so that by Proposition \ref{vwmv} we can assert that $W\in L^{\infty}_{\mathrm{loc}}(M)$. Let $R>0$ be fixed. Given any $ \varepsilon>0 $, there exists $R_\varepsilon>0$ such that $ |R-R_\varepsilon|<\varepsilon $ and $ \Omega_{R_\varepsilon,\varepsilon} $ is a regular precompact domain of $ M $, the sets $ \{ \Omega_{R,\varepsilon} \} $ having been introduced in Subsection \ref{regdist}. This is possible since the set of regular values of the exhaustion function $\mathcal E_{\varepsilon}$ has full measure. Let $ \left\{ \overline W\!_R 	\right\} $ be the supersolutions provided by Lemma \ref{nlem}. In view of their explicit expression \eqref{expl-form}, it is apparent that there exists $ \overline{\varepsilon}=\overline{\varepsilon}(R)>0 $ so small that
\begin{equation}\label{WeW}
 \| W \|_{L^\infty\left(B_{R+2\varepsilon}(o)\right)}  \le \overline{W}\!_{R+2\varepsilon}(R-2\varepsilon) \qquad \forall \varepsilon \in (0,\overline{\varepsilon}) \, .
\end{equation}
Now let us fix any $ \varepsilon \in (0,\overline{\varepsilon}) $. By arguing as in the proof of Proposition \ref{lem:comp} and recalling \eqref{neq8}, we can construct another regular precompact domain $ \hat{\Omega} $ such that $ W $ has a well-defined trace on $\partial \hat \Omega$ and
\begin{equation}\label{compact}
\overline{B}_{R-2\varepsilon}(o) \Subset \hat{\Omega} \subset \overline{\Omega}_{R_\varepsilon,\varepsilon} \Subset B_{R+2\varepsilon}(o) \, .
\end{equation}
Hence, thanks to \eqref{WeW} and the fact that $ \overline{W} $ is radially increasing, we can infer that
$$
 W \le \|W\|_{L^\infty(B_{R+2\varepsilon}(o))} \le \overline{W}\!_{R+2\varepsilon} < \infty  \qquad \text{on } \partial \hat{\Omega} \, .
$$
As a result, we have that $\overline{W}\!_{R+2\varepsilon}$ is a bounded supersolution and $W$ is a bounded subsolution to the following Dirichlet problem:
\begin{equation*}
\begin{cases}
\Delta U = U^p & \text{in } \hat{\Omega}  \, , \\
U = \overline{W}\!_{R+2\varepsilon}|_{\partial \hat \Omega} & \text{on } \partial\hat{\Omega} \, ,
\end{cases}
\end{equation*}
so that in particular their difference satisfies
\begin{equation}\label{very-weak-elliptic}
\int_{\hat \Omega} \left[ \left( \overline{W}\!_{R+2\varepsilon} - W \right) \Delta \varphi - \left( \overline{W}\!_{R+2\varepsilon}^{\,\,p} - W^p \right) \varphi \right] {d}\mu  \le 0
\end{equation}
for every nonnegative smooth function $ \varphi $ that vanishes on $\partial \hat \Omega$.  Starting from \eqref{very-weak-elliptic} and taking advantage of a duality argument similar to the one carried out in the proof of Lemma \ref{loc-comp} (even simpler actually), we deduce that $  W \le \overline{W}\!_{R+2\varepsilon}  $ almost everywhere in $ \hat \Omega $. Recalling \eqref{compact}, for every $ \varepsilon \in (0,\overline{\varepsilon})   $ we thus obtain
\[
W(x) \le  \overline W\!_{R+2\varepsilon}(x)  \qquad \text{for a.e.~$ x $ in } B_{R-2\varepsilon}(o) \, ,
\]
whence, upon letting $ \varepsilon \downarrow 0 $,
\[
W(x) \le  \overline W\!_{R}(x)  \qquad \text{for a.e.~$ x $ in } B_{R}(o) \, .
\]
The thesis then follows by letting $ R \to \infty $, as $ \left\{ \overline{W}\!_R \right\} $ converges locally uniformly to zero in view of \eqref{eq1c}.

\noindent\textit{(ii)} Let $ W \in L^p_{\mathrm{loc}}(M) $ satisfy \eqref{eqn1}. One can reason, for instance, as in \cite[Lemma 2]{Bre}. Indeed, Kato's inequality ensures that $ W^+ $ fulfills \eqref{eqn1-subsol} and therefore is a distributional subsolution to \eqref{eqn1}. By \textit{(i)}  we then deduce that $ W^+ = 0 $. Upon applying the same argument to $ W^- $ it follows that also $ W^- = 0 $, whence $W$ is identically zero.
\end{proof}

\section{Proof of the uniqueness results for the fast diffusion equation}\label{uniq}

We are finally in position to prove Theorems \ref{tuni1} and \ref{tuni2}. The arguments we exploit are inspired from \cite{GIM} and take advantage of the crucial nonexistence results for \eqref{eqn1-subsol} established in Section \ref{Ellptic}.

\begin{proof}[Proof of Theorem \ref{tuni1}]
Given $ t_0>0 $, consider the following function:
$$
W(x) := \int_0^{t_0} \left|u(x,s)^{m}-v(x,s)^{m}\right| e^{-s}  \, {d}s \qquad \forall x \in M \, .
$$
By assumption we know that $ \left| u(\cdot,t) - v(\cdot,t) \right| \to 0 $ in $ L^1_{\mathrm{loc}}(M)  $ as $ t \to 0^+ $, whence in particular
\begin{equation}\label{eq20}
\lim_{t\to 0^+}\int_{M} \left|u(x,t)-v(x,t)\right| \psi(x) \, {d}\mu = 0
\end{equation}
for every $\psi\in C^\infty_c(M)$. If we multiply \eqref{eq14} by $ e^{-t}$ and integrate by parts in $ (0,t_0) $, thanks to \eqref{eq20} we end up with
\begin{equation}\label{eq14-bis}
e^{-t_0} \left| u(x,t_0)-v(x,t_0) \right| + \int_0^{t_0} \left|  u(x,s) - v(x,s) \right| e^{-s} \, {d}s \le \Delta W (x) \qquad \text{in } \mathcal D'(M)\, ,
\end{equation}
which trivially implies
\begin{equation*}\label{eq140}
 \int_0^{t_0} \left|  u(x,s) - v(x,s) \right| e^{-s} \, {d}s \le \Delta W (x) \qquad \text{in } \mathcal D'(M)\,.
\end{equation*}
By applying H\"{o}lder's inequality to the left-hand side, it follows that
\begin{equation}\label{eq140-b}
 \left( \int_0^{t_0} \left|  u(x,s) - v(x,s) \right|^m e^{-s} \, {d}s \right)^{\frac{1}{m}} \left( 1-e^{-t_0} \right)^{-\frac{1-m}{m}} \le \Delta W (x) \qquad \text{in } \mathcal D'(M) \, .
\end{equation}
Recalling the elementary inequality
\begin{equation*}\label{num-m}
2^{m-1} \left|a^{m}-b^{m}\right| \le  |a-b|^m \qquad \forall a,b \in \mathbb{R} \, ,
\end{equation*}
from \eqref{eq140-b} we further deduce that
\begin{equation*}\label{eq141}
 \left( \int_0^{t_0} \left| u(x,s)^{m}-v(x,s)^{m} \right| e^{-s} \, {d}s \right)^{\frac{1}{m}} \underbrace{\left( 2-2e^{-t_0} \right)^{-\frac{1-m}{m}}}_{=:\alpha>0} \le \Delta W (x) \qquad \text{in } \mathcal D'(M) \, ,
\end{equation*}
namely $ W $ satisfies
\begin{equation*}\label{eq142}
\Delta W \ge \alpha \, W^{\frac{1}{m}} \qquad \text{in } \mathcal D'(M) \, .
\end{equation*}
Hence, thanks to Theorem \ref{tell}\textit{(i)} with $ p=1/m $, we can assert that $W = 0$. The thesis then follows in view of the arbitrariness of $ t_0 $.
\end{proof}

\begin{proof}[Proof of Theorem \ref{tuni2}]
By virtue of Theorem \ref{distr-sol}, we know that $ u \ge \underline{u} $. In particular, we can infer that \eqref{eq14} still holds with $ v\equiv\underline{u} $ by merely using the fact that both $ u $ and $ \underline{u} $ are distributional solutions (i.e.~we do not need them to be strong). Hence, thanks to a standard time cut-off argument, by proceeding as in the proof of Theorem \ref{tuni1} we obtain the inequality
\begin{equation}\label{eq180}
\begin{aligned}
& \, e^{-t_0} \left[ u(x,t_0)-\underline u (x,t_0) \right] - e^{-\tau} \left[ u(x,\tau)-\underline u(x,\tau) \right] + \int_\tau^{t_0} \left[  u(x,s) - \underline u(x,s) \right] e^{-s} \, {d}s \\
\le & \, \Delta \int_\tau^{t_0} \left[ u(x,s)^{m} -\underline u(x,s)^{m} \right] e^{-s} \, {d}s \qquad \text{in } \mathcal{D}'(M) \, ,
\end{aligned}
\end{equation}
valid for almost every $ t_0,\tau>0 $ with $ t_0 > \tau $. If we set
$$
W(x) := \int_0^{t_0} \left[ u(x,s)^{m}-\underline u(x,s)^{m} \right] e^{-s} \, {d}s \qquad \forall x \in M
$$
and let $ \tau \to 0^+ $ in \eqref{eq180}, using \eqref{eq12c}, we end up with the analogue of \eqref{eq14-bis} with $ v \equiv \underline{u} $. We can therefore repeat exactly the same passages as above, which lead to $W=0$, so that the thesis follows again in view of the arbitrariness of $ t_0 $.
\end{proof}

\noindent {\sc Acknowledgment.} The authors were partially supported by the PRIN Project  ``Direct and Inverse Problems for Partial Differential Equations: Theoretical Aspects and Applications'' (grant no.~201758MTR2, MIUR, Italy). M.M. and F.P. were also supported by the GNAMPA Projects ``Analytic and Geometric Problems Associated to Nonlinear Elliptic and Parabolic PDEs'', ``Existence and Qualitative Properties for Solutions of Nonlinear Elliptic and Parabolic PDEs''  and ``Differential Equations on Riemannian Manifolds and Global Analysis''. The authors thank the GNAMPA group of the Istituto Nazionale di Alta Matematica (INdAM, Italy).


\begin{thebibliography}{999}

\bibitem{AK} F. Albiac, N. Kalton, ``Topics in Banach Space Theory''. Second Edition. Graduate Texts in Mathematics, 233. Springer, Cham, 2016.


\bibitem{BMMP} C. Bandle, P. Mastrolia, D.D. Monticelli, F. Punzo, \emph{On the stability of solutions of semilinear elliptic equations with Robin boundary conditions on Riemannian manifolds}, SIAM J. Math. Anal. \textbf{48} (2016), 122--151.

\bibitem{BS} D. Bianchi, A. Setti, {\em Laplacian cut-offs, porous and fast diffusion on manifolds and other applications}, Calc. Var. Partial Differential Equations \textbf{57} (2018), Art. 4, 33 pp.

\bibitem{BG} L. Boccardo, T. Gallou\"{e}t, {\em Nonlinear elliptic and parabolic equations involving measure data}, J. Funct. Anal. {\bf 87} (1989), 149--169.

\bibitem{BGV} M. Bonforte, G. Grillo, J.L. V\'azquez, {\em Fast diffusion flow on manifolds of nonpositive curvature}, J. Evol. Equ. {\bf 8} (2008), 99--128.


\bibitem{BV2} M. Bonforte, J.L. V\'azquez, {\em Positivity, local smoothing, and Harnack inequalities for very fast diffusion equations}, Adv. Math. \bf 223 \rm (2010), 529--578.

\bibitem{Bre} H. Br\'ezis, \emph{Semilinear equations in $ \mathbb{R}^N $ without condition at infinity}, Appl. Math. Optim. \textbf{12} (1984), 271--282.

\bibitem{DDS} P. Daskalopoulos, M. del Pino, N. Sesum, \emph{Type II ancient compact solutions to the Yamabe flow},
J. Reine Angew. Math. \bf 738 \rm (2018), 1--71.

\bibitem{Fed} H. Federer, \emph{Curvature measures}, Trans. Amer. Math. Soc. \textbf{93} (1959), 418--491.

\bibitem{Foote} R.L. Foote, \emph{Regularity of the distance function}, Proc. Amer. Math. Soc. \textbf{92} (1984), 153--155.

\bibitem{GW} R.E. Greene, H. Wu, \emph{$C^\infty$ approximations of convex, subharmonic, and plurisubharmonic functions}, Ann. Sci. \'Ecole Norm. Sup. \textbf{12} (1979), 47--84.

\bibitem{Gri} A. Grigor'yan, \emph{Analytic and geometric background of recurrence and non-explosion of the Brownian motion on Riemannian manifolds}, Bull. Amer. Math. Soc. (N.S.) \textbf{36} (1999), 135--249.

\bibitem{GIM} G. Grillo, K. Ishige, M. Muratori, {\em Nonlinear characterizations of stochastic completeness}, to appear on J. Math. Pures Appl., preprint arXiv: \url{https://arxiv.org/abs/1806.03105}.

\bibitem{GM1} G. Grillo, M. Muratori, \emph{Radial fast diffusion on the hyperbolic space}, Proc. London Math. Soc. \bf 109 \rm (2014), 283--317.

\bibitem{GM} G. Grillo, M. Muratori, \emph{Smoothing effects  for the porous medium equation on Cartan-Hadamard manifolds},  Nonlinear Anal. \bf 131 \rm (2016), 346--362.

\bibitem{GMP} G. Grillo, M. Muratori, F. Punzo, \emph{The porous medium equation with large initial data on negatively curved Riemannian manifolds}, J. Math. Pures Appl. \bf 113 \rm (2018), 195--226.

\bibitem{GMP2} G. Grillo, M. Muratori, F. Punzo, \emph{The porous medium equation with measure data on negatively curved Riemannian manifolds}, J. Eur. Math. Soc. (JEMS) \bf 20 \rm (2018), 2769--2812.

\bibitem{GMV} G. Grillo, M. Muratori, J.L. V\'azquez, \emph{The porous medium equation on Riemannian manifolds with negative curvature. The large-time behaviour}, Adv. Math. \bf 314 \rm (2017), 328--377.

\bibitem{GMV-MA} G. Grillo, M. Muratori, J.L. V\'azquez, \emph{The porous medium equation on Riemannian manifolds with negative curvature: the superquadratic case}, Math. Ann. \textbf{373} (2019), 119--153.


\bibitem{HP} M.A. Herrero, M. Pierre, \emph{The Cauchy problem for $u_t = \Delta u^m$ when $0<m<1$}, Trans. Amer. Math. Soc. \textbf{291} (1985), 145--158.



\bibitem{Kato} T. Kato, \emph{Schr\"odinger operators with singular potentials}, Israel J. Math. \textbf{13} (1972), 135--148.

\bibitem{K}  J.B. Keller, \emph{On solutions of $\Delta u = f(u)$}, Comm. Pure Appl. Math. \textbf{10} (1957), 503--510.

\bibitem{Lee} J.M. Lee, ``Introduction to Smooth Manifolds''. Second Edition. Graduate Texts in Mathematics, 218. Springer, New York, 2013.

\bibitem{LS} P. Li, R. Schoen, {\em $L^p$ and mean value properties of subharmonic functions on Riemannian manifolds}, Acta Math. {\bf 153} (1984), 279--301.

\bibitem{LT} P.L. Lions, G. Toscani, \emph{Diffusive limit for finite velocity Boltzmann kinetic models}, Rev. Mat. Iberoamericana \bf 13 \rm (1997), 473--513.

\bibitem{MRS} P. Mastrolia, M. Rigoli, A.G. Setti, ``Yamabe-type Equations on Complete, Noncompact Manifolds''. Progress in Mathematics, 302. Birkh\"auser/Springer Basel AG, Basel, 2012.

\bibitem{M} K. Motomiya, \emph{On functions which satisfy some differential inequalities on Riemannian manifolds}, Nagoya Math. J. \bf 81 \rm (1981), 57--72.

\bibitem{O} R. Osserman, \emph{On the inequality $ \Delta u \ge f(u) $}, Pacific J. Math. \textbf{7} (1957), 1641--1647.

\bibitem{PZ} L.A. Peletier, J.N. Zhao, \emph{Large time behaviour of solutions of the porous media equation with absorption: the fast diffusion case}, Nonlinear Anal. \textbf{17} (1991), 991--1009.

\bibitem{PRS} S. Pigola, M. Rigoli, A.G. Setti, {\em Maximum principles on Riemannian manifolds and applications}, Mem. Amer. Math. Soc. \bf 174 \rm (2005).

\bibitem{V} J.L. V{\'a}zquez, {``The Porous Medium Equation. Mathematical Theory''}. Oxford Mathematical Monographs. The Clarendon Press, Oxford University Press, Oxford, 2007.

\bibitem{V2} J.L. V{\'a}zquez, {\em Fundamental solution and long time behavior of the porous medium equation in hyperbolic space}, J. Math. Pures Appl. \textbf{104} (2015), 454--484.

\end{thebibliography}
\end{document}